\documentclass[a4paper]{amsart}

\usepackage{amssymb}

\usepackage[alphabetic]{amsrefs}

\usepackage{a4wide,enumerate,graphicx,amsthm,esint}

\usepackage{hyperref}

\usepackage{amscd}

\usepackage{tikz}
\usetikzlibrary{cd}

\newtheorem{Satz}{Theorem}[section]
\newtheorem{Proposition}[Satz]{Proposition}
\newtheorem{Lemma}[Satz]{Lemma}
\newtheorem{Corollary}[Satz]{Corollary}

\newtheorem*{Hauptsatz}{Main Theorem}

\theoremstyle{definition}
\newtheorem{Definition}[Satz]{Definition}
\newtheorem{Bemerkung}[Satz]{Remark}

\theoremstyle{remark}
\newtheorem*{claim}{Claim}

\newcommand{\Z}{\mathbb{Z}}
\newcommand{\R}{\mathbb{R}}
\newcommand{\Q}{\mathbb{Q}}
\newcommand{\Hy}{\mathbb{H}}
\newcommand{\p}{\partial}
\newcommand{\ti}{\widetilde}
\newcommand{\ov}{\overline}
\newcommand{\e}{\varepsilon}
\newcommand{\Ha}{\mathcal{H}}
\newcommand{\U}{\overline{U}}

\newcommand{\NPSC}{\rm {NPSC^+}}

\newcommand{\pr}{{\rm pr}}

\DeclareMathOperator{\vol}{vol}

\DeclareMathOperator{\dive}{div}
\DeclareMathOperator{\Sc}{Sc}

\DeclareMathOperator{\dist}{dist}
\DeclareMathOperator{\diam}{diam}
\DeclareMathOperator{\area}{area}
\DeclareMathOperator{\wid}{width}

\DeclareMathOperator{\fil}{FillRad}
\DeclareMathOperator{\Lip}{Lip}
\DeclareMathOperator{\arcoth}{arcoth}

\begin{document}

\title{Scalar and mean curvature comparison via $\mu$-bubbles}

\author{Daniel R\"ade}

\address{Universit\"at Augsburg, Institut f\"ur Mathematik, 86135 Augsburg, Germany}

\email{daniel.raede@math.uni-augsburg.de}

\begin{abstract}
Following ideas of Gromov we prove scalar and mean curvature comparison results for Riemannian bands with lower scalar curvature bounds in dimension $n\leq7$. The model spaces we use are warped products over scalar-flat manifolds with $\log$-concave warping functions. 
\end{abstract}



\maketitle

\section{Introduction}

This article is concerned with a class of manifolds called \emph{bands}. While Gromov gives a very general definition in \cite{Gro18}*{Section 2}, the following is enough for our purposes: 

\begin{Definition}\label{def:band}
A \emph{band} is a connected compact manifold $X$ together with a decomposition $$\p X=\p_-X\sqcup\p_+X,$$ where $\p_\pm X$ are (non-empty) unions of boundary components.
If $X$ is equipped with a Riemannian metric $g$, we call $(X,g)$ a \emph{Riemannian band} and denote by $\wid(X,g)$ the distance (with respect to $g$) between $\p_-X$ and $\p_+X$.
\end{Definition}

\begin{Bemerkung}
The standard example of a band is the cylinder $Y\times [-1,1]$, where $Y$ is a closed manifold. Such bands are called \emph{trivial} in this article.
\end{Bemerkung}

\begin{Definition}\label{def:bandmap}
A continuous map $f:X\to X'$ between two bands is called a \emph{band map} if it maps $\p_- X$ to $\p_- X'$ and $\p_+ X$ to $\p_+X'$. 
\end{Definition}

Let $X^{\geq2}$ be band and $\sigma$ be any real number. By Gromov's $h$-principle, there is a Riemannian metric $g$ on $X$ with scalar curvature $Sc(X,g)\geq\sigma$ and the space of all such metrics is contractible.

To encounter interesting phenomena, similar to the closed case, one needs to impose boundary conditions for the metric, which are typically phrased in terms of mean curvature of the boundary.
For an overview and some recent developments we refer to \cites{BHa20, BMN11, CaLi19}.

\begin{Bemerkung} \label{rem:convention1}
Let $(X,g)$ be a Riemannian manifold with boundary $\p X\neq\emptyset$. In this article $H(\p X,g)$ denotes the trace of the second fundamental form of $\p X$ with respect to the inner unit normal vector field.
With this convention the unit sphere $S^{n-1}\subset\R^n$ has mean curvature $(n-1)$. 
\end{Bemerkung}

We recall three results for the trivial band $T^{n-1}\times[-1,1]$. The first one is due to Gromov and Lawson and follows directly from \cite{GroLa80}*{Theorem A, Theorem 5.7}.

\begin{Satz}[\cite{GroLa80}]\label{cor:mean}
Let $X=T^{n-1}\times[-1,1]$ and $g$ be a Riemannian metric on $X$ with $\Sc(X,g)\geq\sigma>0$. If $H(\p X,g)$ denotes the mean curvature of the boundary, then $\inf_{x\in\p X}H(\p X,g)(x)<0$.
\end{Satz}

The proof is based on the fact that one could produce an impossible metric with positive scalar curvature on the doubled manifold $T^{n-1}\times S^1=T^n$, if the mean curvature were nonnegative.
The second result is due to Gromov \cite{Gro18}, who realized that the distance between the two boundary components in $T^{n-1}\times[-1,1]$ is bounded from above in terms of a lower scalar curvature bound.

\begin{Satz}[\cites{Ce 20,Gro18, RZ19}]\label{thm:torusband}
Let $X=T^{n-1}\times[-1,1]$. If $g$ is a Riemannian metric on $X$ with $\Sc(X,g)\geq\sigma>0$, then 
$$
\wid(X,g)=\dist_g\left(T^{n-1}\times\{-1\},T^{n-1}\times\{1\}\right)\leq 2\pi\sqrt{\frac{n-1}{\sigma n}}.
$$
\end{Satz}

Gromov proved Theorem \ref{thm:torusband} in \cite{Gro18}*{Section 2} for $n\leq 7$ using the minimal hypersurface approach of Schoen and Yau \cite{SY79} combined with a symmetrization argument inspired by Fischer-Colbrie and Schoen \cite{FCS}.

Subsequently Zeidler \cite{RZ19}*{Theorem 1.4} and Cecchini \cite{Ce20}*{Theorem D} proved Theorem \ref{thm:torusband} in all dimensions using Dirac operator methods.
The estimate is sharp but equality can not be attained.

Based on ideas of Gromov \cite{Gro19}*{Section 5.5}, Cecchini and Zeidler \cite{CeZei21}*{Theorem 7.6} showed that Theorem \ref{cor:mean} and Theorem \ref{thm:torusband} are connected by a scalar and mean curvature comparison principle for Riemannian metrics on $T^{n-1}\times[-1,1]$.

\begin{Satz}[\cite{CeZei21}]\label{thm:example}
For $n$ odd let $X=T^{n-1}\times[-1,1]$ and $g$ be a Riemannian metric on $X$. If 
\begin{itemize}
\item[$\triangleright$] $\Sc(X,g)\geq n(n-1),$
\item[$\triangleright$] $H(\p_\pm X,g)\geq\mp (n-1)\tan\left(\frac{n\ell_\pm}{2}\right)$ for some $-\frac{\pi}{n}<\ell_-<\ell_+<\frac{\pi}{n}$,
\end{itemize} 
then $\wid(X,g)\leq \ell_+-\ell_-$.
\end{Satz}

\begin{Bemerkung}\label{rem:explanation}
Let $g_0$ be a scalar flar Riemannian metric on $T^{n-1}$.
Consider the warped product
$$(M,g_\varphi)=\left(T^{n-1}\times[\ell_-,\ell_+], \varphi^2(t)g_0+dt^2\right),$$
where $\varphi(t)=\cos(\frac{nt}{2})^{\frac{2}{n}}$. Standard results for warped products (see \ref{section:basics}) imply $\Sc(M,g_\varphi)=n(n-1)$ and $H(\p_\pm M,g_\varphi)=\mp (n-1)\tan\left(\frac{n\ell_\pm}{2}\right)$.
For $n$ odd let $X=T^{n-1}\times[-1,1]$ and $g$ be a Riemannian metric on $X$.
Theorem \ref{thm:example} is a comparison result in the following sense:
If $\Sc(X,g)\geq\Sc(M,g_\varphi)$ and $H(\p_\pm X,g)\geq H(\p_\pm M,g_\varphi)$, then $\wid(X,g)\leq\wid(M,g_\varphi)$.
\end{Bemerkung}

We quickly show how Theorem \ref{cor:mean} and Theorem \ref{thm:torusband} are implied by Theorem \ref{thm:example} for $n$ odd.

\begin{proof}[Proof of Theorem \ref{cor:mean}]
Suppose $H(\p X,g)\geq0$. 
By rescaling the metric we can assume that $\Sc(X,g)\geq n(n-1)$.
Since $X$ is compact, $\wid(X,g)$ is a positive real number.
Let $0=\ell_-<\ell_+<\wid(X,g)$.
By Theorem \ref{thm:example} we conclude that $\wid(X,g)\leq \ell_+-\ell_-=\ell_+<\wid(X,g)$.
This is a contradiction.
\end{proof}

\begin{proof}[Proof of Theorem \ref{thm:torusband}]
By rescaling the metric we can assume that $\Sc(X,g)\geq n(n-1)$.
Since $X$ is compact, there are constants $c_-$ and $c_+$ such that $H(\p_\pm X,g)\geq c_\pm$.
Since $(n-1)\tan\left(\frac{nt}{2}\right)\to-\infty$ as $t\to-\pi/n$ and $-(n-1)\tan\left(\frac{nt}{2}\right)\to-\infty$ as $t\to\pi/n$ we can find $-\pi/n<\ell_-<\ell_+<\pi/n$ with
$$c_\pm\geq\mp (n-1)\tan\left(\frac{n\ell_\pm}{2}\right).$$
By Theorem \ref{thm:example} we conclude that $\wid(X,g)\leq \ell_+-\ell_-<\frac{2\pi}{n}$.
\end{proof}

In this article we establish some general results (see Section \ref{section:results}) regarding scalar and mean curvature comparison of Riemannian bands with warped products over closed Riemannian manifolds.
Our work is inspired by \cite{CeZei21}, where Cecchini and Zeidler prove a number of such comparison results for certain classes of spin bands using Dirac operator methods. 
However, instead of relying on the Dirac operator, we follow ideas of Gromov, which already appear in \cite{Gro96}*{Section 5$\frac{5}{6}$} and are developed further in \cite{Gro18}*{Section 9} and \cite{Gro19}*{Section 5}, and use a version of the minimal surface approach involving so called $\mu$-bubbles (see Section \ref{section:bubbles}).

To ensure regularity of these $\mu$-bubbles our results are limited to dimension $n\leq7$. Apart from that we have the advantage that we do not need to assume our bands to be spin and odd dimensional, as is done throughout \cite{CeZei21}. Consequently we can work in full generality and establish optimal results towards Gromov's band width conjecture and Rosenberg's $S^1$-stability conjecture \cite{Ro07}*{Conjecture 1.24} in dimensions $n\in \{6,7\}$ (see Corollary \ref{cor:Rosenberg} and Remark \ref{rem:Rosenberg}).

Furthermore, the comparison principle we present provides a general framework, in the context of which the  connections between previous results, regarding Riemannian bands with lower scalar curvature bounds, become apparent.

\section{Comparison Results}\label{section:results} In the following we formulate our main result concerning scalar and mean curvature comparison of Riemannian bands with warped products.

\begin{Definition}\label{def:sep}
Let $X$ be a band. We say that a closed embedded hypersurface $\Sigma\subset X$ \emph{separates} $\p_-X$ and $\p_+X$ if no connected component of $X\backslash\Sigma$ contains a path $\gamma:[0,1]\to X$ with $\gamma(0)\in\p_-X$ and $\gamma(1)\in\p_+X$.
Furthermore $\Sigma$ \emph{properly separates} $\p_-X$ and $\p_+X$ if every connected component of $\Sigma$ can be connected to both $\p_+ X$ and $\p_- X$ inside $X \setminus \Sigma$.

\end{Definition}

\begin{Definition}\label{def:logconcavealt}
A smooth function $\varphi:[a,b]\to \R_+$ is called $\log$-\emph{concave} if
$$\frac{d^2}{dt^2}\log(\varphi)(t)=\left(\frac{\varphi'(t)}{\varphi(t)}\right)'\leq0$$
for all $t\in[a,b]$.
If the inequality is strict we say that $\varphi$ is \emph{strictly $\log$-concave}. In case of equality we say that $\varphi$ is $\log$-\emph{constant}.
\end{Definition}

\begin{Definition}\label{def:model}
Let $(N,g_N)$ be a closed Riemannian manifold with constant scalar curvature. 
A warped product 
$$\left(M,g_\varphi)=(N\times[a,b],\varphi^2(t)g_N+dt^2\right)$$
with warping function $\varphi:[a,b]\to\R_+$ is called a \emph{model space} if $\Sc(M,g_\varphi)$ is constant and $\varphi$ is strictly $\log$-concave or $\log$-constant.
\end{Definition}

\begin{Hauptsatz}\label{MainTheorem}
Let $n\leq7$ and $X^n$ be an oriented band with the property that no closed embedded hypersurface $\Sigma$, which separates $\p_-X$ and $\p_+X$, admits a metric with positive scalar curvature. Let $g$ be a Riemannian metric on $X$ and $(M^n,g_\varphi)$ be a model space over a scalar flat base with warping function $\varphi:[a,b]\to\R_+$. If 
\begin{itemize}
\item[$\triangleright$] $\Sc(X,g)\geq\Sc(M,g_\varphi)$,
\item[$\triangleright$] $H(\p_\pm X,g)\geq H(\p_\pm M,g_\varphi)$,
\end{itemize}
we distinguish two cases:
\begin{enumerate}
\item \label{itm:A} If $\varphi$ is strictly $\log$-concave, then $\wid(X,g)\leq\wid(M,g_\varphi)$.
\item \label{itm:B} If $\varphi$ is $\log$-constant, then $(X,g)$ is isometric to a warped product
$$\left(\hat{N}\times[c,d], \varphi^2g_{\hat{N}}+dt^2 \right),$$
where $(\hat{N},g_{\hat{N}})$ is a closed Ricci flat Riemannian manifold.
\end{enumerate}
\end{Hauptsatz}

\begin{Bemerkung}\label{rem:rigidity}
It is expected that part \eqref{itm:A} of the Main Theorem is rigid as well ie for $\varphi$ strictly $\log$-concave we have $\wid(X,g)=\wid(M,g_\varphi)$ if and only if $(X,g)$ is isometric to a warped product 
$$\left(\hat{N}\times[a,b], \varphi^2g_{\hat{N}}+dt^2 \right),$$
where $(\hat{N},g_{\hat{N}})$ is a closed Ricci flat Riemannian manifold. For spin bands with $\hat{A}(\p_-X)\neq0$ this holds true by work of Cecchini and Zeidler \cite{CeZei21}*{Theorem 8.3, Theorem 9.1}. In  our case there are some obstacles yet to be overcome (see Remark \ref{rem:weakness}). On the other hand the $\log$-constant case ie part \eqref{itm:B} of the Main Theorem is not treated in \cite{CeZei21}.
\end{Bemerkung}

\begin{Bemerkung}
Even if rigidity in \eqref{itm:A} can be established, the two parts of the Main Theorem have to be treated separately, as the width of the band plays a role only in the strictly $\log$-concave case. We point out that we have no control over the width of the band in part \eqref{itm:B} ie the $\log$-constant case and $X$ can be isometric to a model space of arbitrary finite width. 
\end{Bemerkung}

\begin{Bemerkung}\label{rem:2d}
It turns out that the Main Theorem holds true for any oriented band X in dimension $n=2$, where the condition that no closed embedded hypersurface admits a positive scalar curvature metric is vacuous. This will become apparent in Section \ref{section:comparisonresults} (see Remark \ref{rem:2d2}).
\end{Bemerkung}

\subsection{Model Spaces and Applications}\label{section:models} 
For this subsection let $X^n$ be an oriented band with the property that no closed embedded hypersurface $\Sigma$, which separates $\p_-X$ and $\p_+X$, admits a metric with positive scalar curvature and $(N^{n-1},g_N)$ be a closed scalar flat Riemannian manifold.
To understand the type of results we can deduce from our Main Theorem we consider five exemplary model spaces.

For $-\frac{\pi}{n}<\ell_-<\ell_+<\frac{\pi}{n}$ and the $\varphi(t)= \cos\left(\frac{nt}{2}\right)^{\frac{2}{n}}$ (strictly $\log$-concave), the warped product
$$(N\times[\ell_-,\ell_+], \varphi^2(t)g_N+dt^2)$$
is a model space with scalar curvature $n(n-1)$. Plugging this into part \eqref{itm:A} yields:

\begin{Satz}\label{thm:a}
Let $n\leq7$ and $g$ be a Riemannian metric on $X$. If
\begin{itemize}
\item[$\triangleright$] $\Sc(X,g)\geq n(n-1)$
\item[$\triangleright$] $H(\p_\pm X,g)\geq\mp (n-1)\tan\left(\frac{n\ell_\pm}{2}\right)$ for some $-\frac{\pi}{n}<\ell_-<\ell_+<\frac{\pi}{n}$,
\end{itemize}
then $\wid(X,g)\leq \ell_+-\ell_-$.
\end{Satz}

Theorem \ref{thm:a} is a general version of Theorem \ref{thm:example}, as it applies to $X=T^{n-1}\times[-1,1]$ (see \ref{section:examples}). As Theorem \ref{thm:example} implied Theorem \ref{thm:torusband}, Theorem \ref{thm:a} implies the following result due to Gromov, which appears in \cite{Gro19}*{Section 3.6}. Note that we upgrade his result to strict inequality.

\begin{Satz}[\cite{Gro19}*{Section 3.6}]\label{thm:groband}
If $n\leq7$ and $g$ is a Riemannian metric on $X$ with $\Sc(X,g)\geq n(n-1)$, then
$\wid(X,g)< \frac{2\pi}{n}.$
\end{Satz}

Furthermore we generalize Theorem \ref{cor:mean}:

\begin{Corollary}\label{cor:a:mean}
If  $n\leq7$  and $g$ is a Riemannian metric on $X$ with $\Sc(X,g)\geq n(n-1)$, then
$$
\inf_{x\in\p_+X}H(\p_+X,x)\ \ +\inf_{x\in\p_-X}H(\p_-X,x)<0.
$$
\end{Corollary}

For $0<\ell_-<\ell_+<\infty$ and $\varphi(t)= t^{\frac{2}{n}}$ (strictly $\log$-concave) the warped product
$$(N\times[\ell_-,\ell_+], \varphi^2(t)g_N+dt^2)$$
is a scalar flat model space. Plugging this into part \eqref{itm:A} yields:

\begin{Satz}\label{thm:b}
Let $n\leq7$ and $g$ be a Riemannian metric on $X$. If
\begin{itemize}
\item[$\triangleright$] $\Sc(X,g)\geq0$
\item[$\triangleright$] $H(\p_\pm X,g)\geq\pm\frac{2(n-1)}{n\ell_\pm}$ for some $0<\ell_-<\ell_+<\infty$,
\end{itemize}
then $\wid(X,g)\leq \ell_+-\ell_-$.
\end{Satz}

\begin{Corollary}\label{cor:b}
If  $n\leq7$ and $g$ is a Riemannian metric on $X$ with $\Sc(X,g)\geq 0$. If $H(\p_+X)>0$, then 
$$
\wid(X,g)<\frac{2(n-1)}{n\left(\inf_{x\in\p_+X}H(\p_+X,x)\right)}.
$$
\end{Corollary}

For $0<\ell_-<\ell_+<\infty$ and $\varphi(t)= \sinh(\frac{nt}{2})^{\frac{2}{n}}$ (strictly $\log$-concave) the warped product
$$(N\times[\ell_-,\ell_+], \varphi^2(t)g_N+dt^2)$$
is a model space with scalar curvature $-n(n-1)$. Plugging this into part \eqref{itm:A} yields:

\begin{Satz}\label{thm:c}
Let $n\leq7$ and $g$ be a Riemannian metric on $X$. If
\begin{itemize}
\item[$\triangleright$] $\Sc(X,g)\geq0$
\item[$\triangleright$] $H(\p_\pm X)\geq\pm(n-1)\coth\left(\frac{n\ell_\pm}{2}\right)$ for some $0<\ell_-<\ell_+<\infty$,
\end{itemize}
then $\wid(X,g)\leq \ell_+-\ell_-$.
\end{Satz}

For $X=T^{n-1}\times[-1,1]$ the following had already been observed by Gromov \cite{Gro18b}*{Section 4}:

\begin{Corollary}\label{cor:c}
Let $n\leq7$ and $g$ be a Riemannian metric on $X$ with $\Sc(X,g)\geq-n(n-1)$. If $H(\p_+X)>n-1$, then 
\begin{displaymath}
\wid(X,g)<\frac{2}{n}\arcoth\left(\frac{1}{n-1}\inf_{x\in\p_+X}H(\p_+X,x)\right).
\end{displaymath}
\end{Corollary}

For $-\infty<\ell_-<\ell_+<\infty$ and $\varphi(t)=1$ ($\log$-constant) the warped product
$$(N\times[\ell_-,\ell_+], g_N+dt^2)$$
is a scalar flat model space. Plugging this into part \eqref{itm:B} yields the following result, which is probably well known to experts although we were not able to find a reference in the literature.

\begin{Satz}\label{thm:d}
Let $n\leq7$ and $g$ be a Riemannian metric on $X$. If 
\begin{itemize}
\item[$\triangleright$] $\Sc(X,g)\geq0$,
\item[$\triangleright$] $H(\p_\pm X,g)\geq0$,
\end{itemize}
then $(X,g)$ is isometric to a product $(\hat{N}\times[c,d], g_{\hat{N}}+dt^2),$ where $(\hat{N},g_{\hat{N}})$ is a closed Ricci flat Riemannian manifold.
\end{Satz}

For $-\infty<\ell_-<\ell_+<\infty$ and $\varphi(t)= \exp(t)$ ($\log$-constant) the warped product
$$(N\times[\ell_-,\ell_+], \varphi^2(t)g_N+dt^2)$$
is a model space with scalar curvature $-n(n-1)$. Plugging this into part \ref{itm:B} yields:

\begin{Satz}\label{thm:e}
Let $n\leq7$ and $g$ be a Riemannian metric on $X$. If 
\begin{itemize}
\item[$\triangleright$] $\Sc(X,g)\geq-n(n-1)$,
\item[$\triangleright$] $H(\p_\pm X,g)\geq\pm (n-1)$,
\end{itemize}
then $(X,g)$ is isometric to a warped product $(\hat{N}\times[c,d], \exp(2t)g_{\hat{N}}+dt^2),$ where $(\hat{N},g_{\hat{N}})$ is a closed Ricci flat Riemannian manifold.
\end{Satz}

\begin{Bemerkung}\label{rem:PMT}
Special versions of Theorem \ref{thm:e} already appear in \cite{Gro96}*{Section 5$\frac{5}{6}$,p. 57-58} and in \cite{Gro18}*{Section 9}, where its relation to the hyperbolic positive mass theorem is explained. Li proved a cubical version of Theorem \ref{thm:e} in \cite{Li20}*{Theorem 1.3}.
\end{Bemerkung}

\subsection{Topological Results}\label{section:examples} The results of the previous subsection apply to oriented bands $X$ with the property that no closed embedded hypersurface $\Sigma$ which separates $\p_-X$ and $\p_+X$ admits a metric with positive scalar curvature.
Gromov provides a list of examples for such bands in \cite{Gro19}*{Section 3.6}, which we expand significantly. 
In particular we prove the following optimal result for trivial bands in dimension $n\geq6$ (see Section \ref{section:topology}).

\begin{Proposition}\label{prop:cobordismalt}
Let $n\geq6$ and $Y^{n-1}$ be a closed connected oriented manifold which does not admit a metric with positive scalar curvature.
If $X=Y\times[-1,1]$, then no closed embedded hypersurface $\Sigma$ which separates $\p_-X$ and $\p_+X$ admits a metric with positive scalar curvature.
\end{Proposition}

In the spin setting we recall an observation by Zeidler \cites{RZ19, RZ20}.

\begin{Proposition}[\cites{RZ19, RZ20}]\label{prop:rosenberg}
Let $n\geq2$ and $Y^{n-1}$ be a closed connected oriented spin manifold with Rosenberg index $\alpha(Y)\neq0\in KO_{n-1}(C^*\pi_1Y)$.
If $X=Y\times[-1,1]$, then no closed embedded hypersurface $\Sigma$ which separates $\p_-X$ and $\p_+X$ admits a metric with positive scalar curvature.
\end{Proposition}

Furthermore we consider a class of bands which are not necessarily trivial.

\begin{Definition}\label{def:NPSC}
A closed connected oriented manifold $Y^{n-1}$ is called $\NPSC$ if it can not be dominated by a manifold which admits a metric with positive scalar curvature.
In other words: if $Z^{n-1}$ is a closed oriented manifold and there exists a continuous map $f:Z\rightarrow Y$ with $\deg(f)\neq0$, then $Z$ does not admit a metric with positive scalar curvature.
\end{Definition} 

\begin{Definition}\label{def:overNPSC}
A connected oriented band $X^n$ is called over-$\NPSC$ if there is a $\NPSC$-manifold $Y^{n-1}$ and a band map $f:X\to Y\times [-1,1]$ with $\deg(f)\neq0$.
\end{Definition}

\begin{Proposition}\label{prop:NPSC}
Let $n\geq3$ and $X$ be a connected oriented over-$\NPSC$ band. 
Then no closed embedded hypersurface $\Sigma$ which separates $\p_-X$ and $\p_+X$ admits a metric with positive scalar curvature.
\end{Proposition}

\begin{Bemerkung}\label{rem:NPSC}
The two classical examples of $\NPSC$-manifolds one should have in mind are \emph{enlargeable} manifolds (compare \cite{GroLa83}*{Theorem 5.8}, \cite{CeSch18}*{Theorem  A} and \cite{GroLa83}*{Proposition 5.7}) as well as \emph{Schoen-Yau-Schick} manifolds (compare \cite{SY79}*{Theorem 1}, \cite{TS98} and \cite{ChLi20}*{Definition 23}).
\end{Bemerkung}

Chodosh and Li \cite{ChLi20}*{Theorem 2} and Gromov \cite{Gro20}*{Section 7} used $\mu$-bubbles to prove that closed aspherical manifolds of dimension $\leq 5$ do not admit metrics with positive scalar curvature. 
We implement Gromov's approach from \cite{Gro19} in the language of \cite{ChLi20} and present a proof of the following in Section \ref{section:last}

\begin{Satz}\label{prop:deg}
All closed connected oriented aspherical 4-manifolds are $\NPSC$.
\end{Satz}

\begin{Bemerkung}
In the first arXiv version of this article there was a mistake in the proof of Theorem \ref{prop:deg}, which was pointed out to us by Otis Chodosh and Chao Li (the missing piece was Proposition \ref{prop:fix}).
In subsequent joint work with Yevgeny Liokumovich they classified sufficiently connected manifolds in dimension $4$ and $5$ which admits a positive scalar curvature metric. Their result implies Theorem \ref{prop:deg} as well (see \cite{ChLiL21}*{Theorem 3}).
\end{Bemerkung}

Combined with Theorem \ref{thm:a} or Theorem \ref{thm:groband} we establish the following two results towards Gromov's band width conjecture \cite{Gro18}*{11.12, Conjecture C}.

\begin{Corollary}\label{cor:Rosenberg}
Let $n\in\{2,3,4,6,7\}$ and $Y^{n-1}$ be a closed connected oriented manifold which does not admit a metric with positive scalar curvature.
If $X=Y\times[-1,1]$ and $g$ is a Riemannian metric on $X$ with $\Sc(X,g)\geq n(n-1)$, then $\wid(X,g)<\frac{2\pi}{n}$.
\end{Corollary}

\begin{Bemerkung}\label{rem:Rosenberg}
We point out that Corollary \ref{rem:Rosenberg} implies the $S^1$-stability conjecture of Rosenberg \cite{Ro07}*{Conjecture 1.24} for closed connected orientable manifolds of dimension 5 or 6.
\end{Bemerkung}

\begin{Corollary}\label{cor:aspherical}
Let $Y^4$ be a closed connected aspherical manifold.
If $X=Y\times[-1,1]$ and $g$ is a Riemannian metric on $X$ with $\Sc(X,g)\geq n(n-1)$, then $\wid(X,g)<\frac{2\pi}{n}$.
\end{Corollary}

\begin{Bemerkung}
In Corollary \ref{cor:aspherical} the manifold $Y$ can be nonorientable. In that case we simply pass to the orientable double cover, which is a closed connected aspherical manifold as well.
\end{Bemerkung}

\subsection{Acknowledgements.} This is part my ongoing doctoral dissertation project at Augsburg University. I am grateful to my advisor Bernhard Hanke for his continued support, Johannes Ebert for his expertise regarding Proposition \ref{prop:cobordism} and Jan Metzger for answering my questions concerning maximum principles. I thank Rudolf Zeidler, Simone Cecchini and Georg Frenck for helpful comments. This work was supported by a doctoral grant from the German Academic Scholarship Foundation.

\section{Warped products} \label{section:warpedproducts}
In this section we recall some facts about warped products and develop a general framework for scalar and mean curvature comparison of Riemannian bands.

\subsection{Basics}\label{section:basics} The following definitions and formulas are standard knowledge.

\begin{Definition}\label{def:warpedproduct}
Let $(N,g_N)$ be a closed Riemannian manifold and $\varphi: (a,b)\rightarrow \R_+$ be a smooth positive function. The warped product over $(N,g_N)$ with warping  function $\varphi$ is
$$(M,g_\varphi):=\left(N\times(a,b),\varphi^2g_N+ dt^2\right).$$
\end{Definition}

The scalar curvature of $(M,g_\varphi)$ is determined by the scalar curvature of $(N,g_N)$ and the warping function $\varphi$. 
The following formula
\begin{equation}\label{eq:scwarped}
\Sc(M,g_\varphi)(p,t)=\frac{1}{\varphi^2(t)}\Sc(N,g_N)(p)-2(n-1)\frac{\varphi''(t)}{\varphi(t)}-(n-1)(n-2)\left(\frac{\varphi'(t)}{\varphi(t)}\right)^2
\end{equation}
is obtained by a straightforward calculation (see also \cite{Gro19}*{Section 2.4}).\\

If we denote $N_t:=N\times\{t\}$ for $t\in(a,b)$ and consider $N_t$ as the boundary of $N\times(a,t]$, then its second fundamental form with respect to the inner unit normal vector field is a diagonal matrix whose entries are all equal to $$\frac{d}{dt}\log(\varphi)(t)=\frac{\varphi'(t)}{\varphi(t)}.$$ 
It follows that $N_t$ is what is called an umbilic hypersurface and its mean curvature is given by
\begin{equation}\label{eq:meancurv}
H(N_t)=(n-1)\frac{\varphi'(t)}{\varphi(t)}=:h_\varphi(t).
\end{equation}
Finally, we rearrange the formula \eqref{eq:scwarped} for the scalar curvature of a warped product in terms of $h_\varphi$:
\begin{equation}\label{eq:warpedODE}
Sc(M,g_\varphi)(p,t)+\frac{n}{n-1}h_\varphi(t)^2+2h_\varphi'(t)=\frac{1}{\varphi^2(t)}Sc(N,g_N)(p).
\end{equation}
This formula, which combines information on scalar and mean curvature of a warped product, is the basis on which we build a comparison principle.

\subsection{Comparison of two Warped Products} As a first step we want to compare two warped products $(M_1,g_{\varphi_1})$ and $(M_2,g_{\varphi_2})$ over the same base manifold $(N,g_N)$. We start with the simplest case, where $\varphi_1$ and $\varphi_2$ have the same domain and $(N,g_N)$ is scalar flat.

\begin{Proposition}\label{prop:basecase}
Let $(N,g_N)$ be a closed scalar flat Riemannian manifold.
Let $\varphi_1:[a,b]\to\R_+$ and $\varphi_2:[a,b]\to \R_+$ be two positive functions. 
If
\begin{itemize}
\item[$\triangleright$] $\Sc(M,g_{\varphi_1})\geq\Sc(M,g_{\varphi_2})$,
\item[$\triangleright$] $H(\p_\pm M,g_{\varphi_1})\geq H(\p_\pm M,g_{\varphi_2})$,
\end{itemize}
then $h_{\varphi_1}=h_{\varphi_2}$ ie equality holds in both conditions.
\end{Proposition}

Even though the content of Proposition \ref{prop:basecase} is geometric, its proof is purely analytic and based  on the following:

\begin{Lemma}\label{lem:warpedODE1}
Let $\varphi_1:[a,b]\to\R_+$ and $\varphi_2:[a,b]\to\R_+$ be two smooth positive functions.
Then $h_{\varphi_1}= h_{\varphi_2}$ if and only if
\begin{itemize}
\item[$\triangleright$] $\frac{n}{n-1}h_{\varphi_1}^2+2h_{\varphi_1}'\leq \frac{n}{n-1}h_{\varphi_2}^2+2h_{\varphi_2}'$,
\item[$\triangleright$]  $h_{\varphi_1}(a)\leq h_{\varphi_2}(a)$ and $h_{\varphi_1}(b)\geq h_{\varphi_2}(b)$.
\end{itemize}
\end{Lemma}

\begin{proof}
The idea is to reduce the statement to a comparison result for the Riccati equation which can be found in \cite[Lemma 4.1]{WB16}.
Consider $\hat{\varphi}_i(t)=\varphi_i\left(2\sqrt{\frac{n-1}{n}}t\right)^{\frac{n}{2}}$ as functions $[\hat{a},\hat{b}]\to\R_+$ where $\hat{a}:=\frac{a\sqrt{n}}{2\sqrt{n-1}}$ and $\hat{b}:=\frac{b\sqrt{n}}{2\sqrt{n-1}}$.
We denote
$$\hat{h}_i(t):=\frac{\hat{\varphi}_i'(t)}{\hat{\varphi}_i(t)}=\sqrt{\frac{n}{n-1}}(n-1)\frac{\varphi_i'\left(2\sqrt{\frac{n-1}{n}}t\right)}{\varphi_i\left(2\sqrt{\frac{n-1}{n}}t\right)}=\sqrt{\frac{n}{n-1}}h_{\varphi_i}\left(2\sqrt{\frac{n-1}{n}}t\right).$$
Then 
$$\hat{h}_i^2(t)+\hat{h}_i'(t)=\frac{n}{n-1}h_{\varphi_i}^2\left(2\sqrt{\frac{n-1}{n}}t\right)+2h_{\varphi_i}'\left(2\sqrt{\frac{n-1}{n}}t\right).$$
Furthermore, if we denote $\kappa_i:=-\hat{h}_i^2(t)-\hat{h}_i'(t)$, we see that $\kappa_2\leq\kappa_1$ and 
$$\hat{\varphi}_i''(t)+\kappa_i\hat{\varphi}_i(t)=0.$$
At this point we are in the situation where we can apply \cite[Lemma 4.1]{WB16} to conclude.

For the convenience of the reader, we repeat the proof here.
Hence
\begin{displaymath}
\begin{split}
0&=\int_{\hat{a}}^t\hat{\varphi}_1(\hat{\varphi}_2''+\kappa_2\hat{\varphi}_2)-(\hat{\varphi}_1''+\kappa_1\hat{\varphi}_1)\hat{\varphi}_2\\
&=(\hat{\varphi}_1\hat{\varphi}_2'-\hat{\varphi}_1'\hat{\varphi}_2)\bigr|_{\hat{a}}^t+\int_{\hat{a}}^t(\kappa_2-\kappa_1)\hat{\varphi}_1\hat{\varphi}_2
\end{split}
\end{displaymath}
and therefore
\begin{equation}\label{eq:integral}\hat{\varphi}_1(t)\hat{\varphi}_2'(t)-\hat{\varphi}_1'(t)\hat{\varphi}_2(t)=\hat{\varphi}_1(\hat{a})\hat{\varphi}_2'(\hat{a})-\hat{\varphi}_1'(\hat{a})\hat{\varphi}_2(\hat{a})+\int_{\hat{a}}^t(\kappa_1-\kappa_2)\hat{\varphi}_1\hat{\varphi}_2.\end{equation}
Now $\hat{\varphi}_1(\hat{a})\hat{\varphi}_2'(\hat{a})-\hat{\varphi}_1'(\hat{a})\hat{\varphi}_2(\hat{a})\geq0$ since $\hat{h}_1(\hat{a})\leq\hat{h}_2(\hat{a})$ and the second term on the right hand side is nonnegative since $\kappa_2\leq\kappa_1$ and $\hat{\varphi}_1\hat{\varphi}_2>0$.
It follows that 
\begin{equation}\label{eq:inequalityode}
\hat{\varphi}_1(t)\hat{\varphi}_2'(t)-\hat{\varphi}_1'(t)\hat{\varphi}_2(t)\geq0\Leftrightarrow\frac{\hat{\varphi}_1'(t)}{\hat{\varphi}_1(t)}\leq\frac{\hat{\varphi}_2'(t)}{\hat{\varphi}_2(t)}\Leftrightarrow \hat{h}_1(t)\leq\hat{h}_2(t)
\end{equation}
for all $t\in[\hat{a},\hat{b}]$.
By \eqref{eq:integral} $\hat{h}_1(\hat{a})=\hat{h}_2(\hat{a})$ if equality holds at $t$.
We can replace $\hat{a}$ by any $t_0\in[\hat{a},t]$ in the argument above since $\hat{h}_1(t_0)\leq\hat{h}_2(t_0)$.
Hence $\hat{h}_1=\hat{h}_2$ on $[\hat{a},t]$ if equality holds at $t$.
Since $\hat{h}_1(\hat{b})\geq\hat{h}_2(\hat{b})$ by assumption and $\hat{h}_1(\hat{b})\leq\hat{h}_2(\hat{b})$ by \eqref{eq:inequalityode}, equality holds at $\hat{b}$. 
Hence $\hat{h}_1=\hat{h}_2$ on $[\hat{a},\hat{b}]$ and therefore $h_{\varphi_1}=h_{\varphi_2}$ on $[a,b]$.\end{proof}

\begin{proof}[Proof of Proposition \ref{prop:basecase}]
Since $(N,g_N)$ is scalar flat \eqref{eq:warpedODE} implies
\begin{displaymath}
\begin{split}
\frac{n}{n-1}h_{\varphi_1}^2(t)+2h_{\varphi_1}'(t)&=-\Sc(M,g_{\varphi_1})(p,t)\\
&\leq-\Sc(M,g_{\varphi_2})(p,t)=\frac{n}{n-1}h_{\varphi_2}^2(t)+2h_{\varphi_2}'(t).
\end{split}
\end{displaymath}
Furthermore we have
$$h_{\varphi_1}(a)=-H(\p_-M,g_{\varphi_1})\leq-H(\p_-M,g_{\varphi_2})= h_{\varphi_2}(a)$$
and
$$h_{\varphi_1}(b)=H(\p_+M,g_{\varphi_1})\geq H(\p_+M,g_{\varphi_2})= h_{\varphi_2}(a)$$
by \eqref{eq:meancurv}. Thus $h_{\varphi_1}=h_{\varphi_2}$ by Lemma \ref{lem:warpedODE1}.
\end{proof}

Next, we allow the warping functions to have different domains.
Let $(N,g_N)$ be a closed scalar flat Riemannian manifold and $\varphi_1:[a,b]\to \R_+$ and $\varphi_2:[c,d]\to\R_+$ two positive functions.

To compare the scalar curvature of the warped products $(M_1,g_{\varphi_1})$ and $(M_2,g_{\varphi_2})$ pointwise, we need to choose a band map $\Phi:M_1\to M_2$.
In this setting the canonical choice is $\Phi={\rm id}_N\times\phi$, where $\phi:[a,b]\to[c,d]$ is given by $t\mapsto \frac{(d-c)}{(b-a)}(t-a)+c$.

To prove a comparison result like Proposition \ref{prop:basecase} we want to apply Lemma \ref{lem:warpedODE1} to the functions $h_{\varphi_1}$ and $\ti{h}_{\varphi_2}=h_{\varphi_2}\circ\phi=h_{\ti{\varphi}_2}$ where $$\ti{\varphi}_2:[a,b]\to\R_+ \ \ t\mapsto\varphi_2(\phi(t))^{\frac{b-a}{d-c}}.$$
Hence we need to ensure that $\Sc(M_1,g_{\varphi_1})(p,t)\geq\Sc(M_2,g_{\varphi_2})(p,\phi(t))$ implies  
$$
\frac{n}{n-1}h_{\varphi_1}^2(t)+2h_{\varphi_1}'(t)\leq\frac{n}{n-1}\ti{h}_{\varphi_2}^2(t)+2\ti{h}_{\varphi_2}'(t).
$$
for all $t\in[a,b]$. By \eqref{eq:warpedODE} this works if
$$h'_{\varphi_2}(\phi(t))\leq\ti{h}_{\varphi_2}'(t)=h'_{\varphi_2}(\phi(t))\phi'(t),$$
which in turn holds true if $h'_{\varphi_2}(\phi(t))=0$ or $h'_{\varphi_2}(\phi(t))<0$ and $\phi'(t)\leq1$ ie $b-a\geq d-c$.

For this reason we consider strictly $\log$-concave or $\log$-constant warping functions in our comparison results.

\begin{Proposition}\label{prop:warpedODE3}
Let $(N,g_N)$ be a scalar flat Riemannian manifold. Let $\varphi_1:[a,b]\to\R_+$ and $\varphi_2:[c,d]\to \R_+$ be two positive functions. 
Consider the warped products $(M_1,g_{\varphi_1})$ and $(M_2,g_{\varphi_2})$ and the map $\phi:[a,b]\to[c,d]$ given by $t\mapsto \frac{(d-c)}{(b-a)}(t-a)+c$.
If $\varphi_2$ is $\log$-constant,
\begin{itemize}
\item[$\triangleright$] $\Sc(M_1,g_{\varphi_1})(p,t)\geq\Sc(M_2,g_{\varphi_2})(p,\phi(t))$,
\item[$\triangleright$] $H(\p M_1,g_{\varphi_1})\geq H(\p M_2,g_{\varphi_2})$,
\end{itemize}
then $h_{\varphi_1}=h_{\varphi_2}\circ\phi$ ie equality holds in both conditions.
\end{Proposition}

\begin{proof}[Proof of Proposition \ref{prop:warpedODE3}] 
Denote $\ti{h}_{\varphi_2}=h_{\varphi_2}\circ\phi:[a,b]\to \R$.
By \eqref{eq:warpedODE}
$$\Sc(M_1,g_{\varphi_1})(p,t)=-\frac{n}{n-1}h_{\varphi_1}^2(t)-2h_{\varphi_1}'(t)$$
as well as
$$\Sc(M_2,g_{\varphi_2})(p,\phi(t))=-\frac{n}{n-1}h_{\varphi_2}^2(\phi(t))-2h_{\varphi_2}'(\phi(t))=-\frac{n}{n-1}\ti{h}_{\varphi_2}^2(t),$$
since $\varphi_2$ is assumed to be $\log$-constant ie $h'_{\varphi_2}=0$. Hence 
\begin{displaymath}
\frac{n}{n-1}h_{\varphi_1}^2(t)+2h_{\varphi_1}'(t)\leq\frac{n}{n-1}\ti{h}_{\varphi_2}^2(t)=\frac{n}{n-1}\ti{h}_{\varphi_2}^2(t)+2\ti{h}'_{\varphi_2}(t).
\end{displaymath}
Furthermore $h_{\varphi_1}(a)\leq\ti{h}_{\varphi_2}(a)$ and $h_{\varphi_1}(b)\geq \ti{h}_{\varphi_2}(b)$ (this follows from \eqref{eq:meancurv} and the assumption on mean curvature).
Now Lemma \ref{lem:warpedODE1} implies $h_{\varphi_1}=\ti{h}_{\varphi_2}$.\end{proof}

\begin{Proposition}\label{prop:warpedODE2}
Let $(N,g_N)$ be a scalar flat Riemannian manifold. Let $\varphi_1:[a,b]\to\R_+$ and $\varphi_2:[c,d]\to \R_+$ be two positive functions. Consider the warped products $(M_1,g_{\varphi_1})$ and $(M_2,g_{\varphi_2})$ and the map $\phi:[a,b]\to[c,d]$ given by $t\mapsto \frac{(d-c)}{(b-a)}(t-a)+c$.
If $\varphi_2$ is strictly $\log$-concave,
\begin{itemize}
\item[$\triangleright$] $\Sc(M_1,g_{\varphi_1})(p,t)\geq\Sc(M_2,g_{\varphi_2})(p,\phi(t))$,
\item[$\triangleright$] $H(\p M_1,g_{\varphi_1})\geq H(\p M_2,g_{\varphi_2})$,
\item[$\triangleright$] $\wid(M_1,g_{\varphi_1})\geq\wid(M_2,g_{\varphi_2})$,
\end{itemize}
then $b-a=d-c$ and $h_{\varphi_1}=h_{\varphi_2}\circ\phi$ ie equality holds in all three conditions.
\end{Proposition}

\begin{proof}[Proof of Proposition \ref{prop:warpedODE2}]
Denote $\ti{h}_{\varphi_2}=h_{\varphi_2}\circ\phi:[a,b]\to \R$.
As before:
\begin{equation}\label{eq:conditionlemma}
\begin{split}
\frac{n}{n-1}h_{\varphi_1}^2(t)+2h_{\varphi_1}'(t)&\leq\frac{n}{n-1}h_{\varphi_2}^2(\phi(t))+2h_{\varphi_2}'(\phi(t))\\
&\leq\frac{n}{n-1}\ti{h}_{\varphi_2}^2(t)+2\ti{h}_{\varphi_2}'(t),
\end{split}
\end{equation}
where we used that $\varphi_2$ is $\log$-concave and $\phi$ is 1-Lipschitz for the last inequality.

Furthermore $h_{\varphi_1}(a)\leq\ti{h}_{\varphi_2}(a)$ and $h_{\varphi_1}(b)\geq \ti{h}_{\varphi_2}(b)$.
If $b-a>d-c$ ie $\phi$ is strictly 1-Lipschitz, the last inequality in \eqref{eq:conditionlemma} would be strict since $h_{\varphi_2}'<0$.
This is impossible because Lemma \ref{lem:warpedODE1} implies $h_{\varphi_1}=\ti{h}_{\varphi_2}$.
\end{proof}

The main goal of this article is to generalize Proposition \ref{prop:warpedODE3} and Proposition \ref{prop:warpedODE2} to allow for the comparison of Riemannian bands with warped products over closed Riemannian manifolds with constant scalar curvature.

\subsection{General Comparison and Structural Maps}
A (pointwise) scalar and mean curvature comparison of two Riemannian bands $(X,g)$ and $(V,\tau)$ involves a choice of band map $\Phi:X\to V$.
The choice does not matter, if $\Sc(V,\tau)$ is constant and $H(\p V,\tau)$ is constant on $\p_-V$ resp. $\p_+V$.

If $(V,\tau)$ is a warped product $(M,g_\varphi)$ over a closed Riemannian manifold $(N,g_N)$ with warping function $\varphi:[a,b]\to\R_+$, the second condition is satisfied as the mean curvature of $\p M=\p_-M\sqcup\p_+M$ with respect to $g_\varphi$ is constant equal to $\pm h_\varphi(a)$ on $\p_\pm(M)$ (see \eqref{eq:meancurv}).

Furthermore, if $\Sc(N,g_N)$ is constant, then $\Sc(M,g_\varphi)(p,t)$ only depends on the $t$-coordinate (see \eqref{eq:warpedODE}) and therefore the scalar curvature comparison between $(X,g)$ and $(M,g_\varphi)$ only depends on $\phi:=\pr_{[a,b]}\circ\Phi:X\to [a,b]$.

For the rest of this article we focus on this situation ie if not further specified $(X^n,g)$ will be a Riemannian band, $(N^{n-1},g_N)$ will be a closed Riemannian manifold with constant scalar curvature and $(M^n,g_\varphi)$ will be a warped product over $(N,g_N)$ with warping function $\varphi:[a,b]\to\R_+$.
To compare $(X,g)$ and $(M,g_\varphi)$ we fix a point $p_0\in N$, choose a band map $\phi:X\to[a,b]$ and define $\Phi:X\to M$ by $x\mapsto (p_0,\phi(x))$.\\

While every choice of $\phi$ enables us to compare the scalar and mean curvature of $(X,g)$ and $(M,g_\varphi)$ pointwise, we need $\phi$ to preserve some geometric structure to prove a comparison result like Proposition \ref{prop:warpedODE3} or Proposition \ref{prop:warpedODE2}.
We denote $h=h_\varphi\circ\phi:X\to \R$ and consider a 'pullback' version of equation \eqref{eq:warpedODE} on $(X,g)$.

\begin{Definition}\label{def:structural}
A band map $\phi:X\to[a,b]$, which is used to compare $(X,g)$ and $(M,g_\varphi)$, is called \emph{structural} if it is smooth and for any smooth hypersurface $\Sigma$ with outward unit normal field $\nu$ which separates $\p_-X$ and $\p_+X$ in $X$  the inequality
\begin{equation}\label{eq:structural}
\Sc(X,g)(x)+\frac{n}{n-1}h^2(x)+2g(\nabla h(x),\nu(x))\geq \frac{1}{\varphi^2(\phi(x))}Sc(N,g_N)
\end{equation}
holds at all points $x\in \Sigma$. 
\end{Definition}


In Section \ref{section:bubbles} we will use $\mu$-bubbles to prove the following Proposition, which will act as a replacement for Lemma \ref{lem:warpedODE1} in the general case:

\begin{Proposition}\label{prop:comparison1}
Let $n\leq7$ and $(X,g)$ be an oriented Riemannian band. Let $(N,g_N)$ be a closed oriented Riemannian manifold with constant scalar curvature and $(M,g_\varphi)$ the warped product over $(N,g_N)$ with warping function $\varphi:[a,b]\to\R_+$.
If there is a structural band map $\phi:X\to[a,b]$ and 
$$H(\p_\pm X,g)>H(\p_\pm M,g_\varphi),$$
there is a hypersurface $\Sigma\subset X$, which separates $\p_-X$ and $\p_+X$ in $X$ with:
$$-\Delta_{\Sigma}+\frac{1}{2}\Sc(\Sigma,g)\geq\frac{1}{2\varphi^2(\phi)}Sc(N,g_N).$$
\end{Proposition}

\begin{Bemerkung}\label{rem:weakness}
From a conceptual perspective one should be able to relax the assumption $H(\p_\pm X,g)>H(\p_\pm M,g_\varphi)$ in Proposition \ref{prop:comparison1} to $H(\p_\pm X,g)\geq H(\p_\pm M,g_\varphi)$ if $\varphi$ is $\log$-concave. 
However, we are only able to do so whenever $\varphi$ is $\log$-constant (see Section \ref{section:CMC}). 
If $\varphi$ is strictly $\log$-concave we can work around certain aspects of the problem but fall short of the desired result.
The reason for this indiscrepancy is the lack of a strong maximum principle for $\mu$-bubbles.
\end{Bemerkung}

In light of Proposition \ref{prop:comparison1} we try to identify situations where there are structural maps to compare $(X,g)$ and $(M,g_\varphi)$.
As in Proposition \ref{prop:warpedODE3} and Proposition \ref{prop:warpedODE2} we assume $\varphi$ to be strictly $\log$-concave or $\log$-constant.

\begin{Lemma}\label{lem:constantstructure}
Let $(X,g)$ be a Riemannian band, $(N,g_N)$ be a closed Riemannian manifold with constant scalar curvature and $(M,g_\varphi)$ be a warped product over $(N,g_N)$ with warping function $\varphi:[a,b]\to\R_+$. If $\varphi$ is $\log$-constant, then any smooth band map $\phi:X\to[a,b]$ such that $\Sc(X,g)(x)\geq\Sc(M,g_\varphi)(p_0,\phi(x))$ is structural.
\end{Lemma}

\begin{proof}
Let $\Sigma$ be a hypersurface which separates $\p_-X$ and $\p_+X$ in $X$. 
If $\phi:X\to[a,b]$ is smooth, and $\Sc(X,g)(x)\geq\Sc(M,g_\varphi)(p_0,\phi(x))$, then
\begin{displaymath}
\Sc(X,g)(x)+\frac{n}{n-1}h^2(x)+2g(\nabla h(x),\nu(x))=\Sc(X,g)(x)+\frac{n}{n-1}h^2(x)
\end{displaymath}
since $\nabla h=0$ ($\varphi$ is $\log$-constant) and
\begin{displaymath}
\begin{split}
\Sc(X,g)(x)+\frac{n}{n-1}h^2(x)&\geq\Sc(M,g_\varphi)(p_0,\phi(x))+\frac{n}{n-1}h_\varphi^2(\phi(x))\\
&=\frac{1}{\varphi^2(\phi(x))}Sc(N,g_N),
\end{split}
\end{displaymath}
which implies that $\phi$ is structural.
\end{proof}

\begin{Lemma}\label{lem:logconcavestructure}
Let $(X,g)$ be a Riemannian band, $(N,g_N)$ be a closed Riemannian manifold with constant scalar curvature and $(M,g_\varphi)$ be a warped product over $(N,g_N)$ with warping function $\varphi:[a,b]\to\R_+$. If 
\begin{itemize}
\item[$\triangleright$] $\phi:X\to[a,b]$ is a smooth 1-Lipschitz band map,
\item[$\triangleright$] $\varphi$ is strictly $\log$-concave and
\item[$\triangleright$] $\Sc(X,g)(x)\geq\Sc(M,g_\varphi)(p_0,\phi(x))$,
\end{itemize}
then $\phi$ is structural. 
\end{Lemma}

\begin{proof}
Let $\Sigma$ be a hypersurface which separates $\p_-X$ and $\p_+X$ in $X$. Since $\varphi$ is strictly $\log$-concave ie $h_\varphi'<0$ and $\phi$ is 1-Lipschitz we have
$$\frac{n}{n-1}h^2(x)+2g(\nabla h(x),\nu(x))\geq\frac{n}{n-1}h_\varphi^2(\phi(x))+2h_\varphi'(\phi(x))|\nabla\phi|\geq \frac{n}{n-1}h_\varphi^2(\phi(x))+2h_\varphi'(\phi(x)).$$
Together with $\Sc(X,g)(x)\geq\Sc(M,g_\varphi)(p_0,\phi(x))$ and \eqref{eq:warpedODE} it follows that
$\phi$ is structural.
\end{proof}

The following can also be found in \cite{Zhu20}*{Lemma 4.1} and \cite{CeZei21}*{Lemma 7.2}.

\begin{Lemma}\label{lem:map}
Let $(X,g)$ be a Riemannian band. If $\wid(X,g)>a-b$, there is a smooth band map $\phi:(X,g)\rightarrow[a,b]$ with $\Lip(\phi)<1$.
\end{Lemma}

A combination of Lemma \ref{lem:logconcavestructure} and Lemma \ref{lem:map} yields:

\begin{Lemma}\label{lem:3condstructure}
Let $(X,g)$ be a Riemannian band, $(N,g_N)$ be a closed Riemannian manifold with constant scalar curvature and $(M,g_\varphi)$ be a warped product over $(N,g_N)$ with warping function $\varphi:[a,b]\to\R_+$. If $\varphi$ is strictly $\log$-concave, $\Sc(M,g_\varphi)$ is constant, and the following holds true
\begin{itemize}
\item[$\triangleright$] $\Sc(X,g)\geq\Sc(M,g_\varphi)$,
\item[$\triangleright$] $\wid(X,g)>\wid(M,g_\varphi)$,
\end{itemize}
there exists a structural band map $\phi:X\to[a,b]$.\qed
\end{Lemma}

\begin{Bemerkung}\label{rem:addition}
If $\Lip(\phi)<1$ in Lemma \ref{lem:logconcavestructure} we get strict inequality in \eqref{eq:structural}. This observation is important as it allows us to obtain strict inequality for the operator in Proposition \ref{prop:comparison1} later (see Remark \ref{rem:sharp1}). In particular this applies to the band map $\phi$ we get from Lemma \ref{lem:3condstructure}.
\end{Bemerkung}

\subsection{Model Spaces}
The notion of a model space (compare Definition \ref{def:model} for scalar and mean curvature comparison of Riemannian bands is motivated by our observations so far. In addition to those we introduced in Section \ref{section:models} one considers Annuli in simply connected space forms.\\
 
Let $(S^n,g_1)\backslash\{p_1,p_2\}$ be the unit $n$-sphere with two antipodal points removed. 
This has constant scalar curvature equal to $n(n-1)$ and can be written as a warped product $$\left(S^{n-1}\times(-\frac{\pi}{2},\frac{\pi}{2}),\cos^2(t)g_1+ dt^2\right),$$
where $(S^{n-1},g_1)$ is the unit sphere in one dimension less. 
Since $\cos(t)$ is strictly $\log$-concave we see that for $-\frac{\pi}{2}<\ell_-<\ell_+<\frac{\pi}{2}$ the warped product
$$\left(S^{n-1}\times[\ell_-,\ell_+],\cos^2(t)g_1+ dt^2\right)$$
is a model space.\\

Let $(\R^n,g_{\rm std})\backslash\{0\}$ be euclidean space with the origin removed.
This is scalar flat and can be written as a warped product 
$$\left(S^{n-1}\times(0,\infty), t^2g_1+ dt^2\right)$$
Since $t$ is stricly $\log$-concave we see that for $0<\ell_-<\ell_+<\infty$ the warped product
$$\left(S^{n-1}\times[\ell_-,\ell_+],t^2g_1+ dt^2\right)$$
is a model space.\\

Let $(\Hy^n,g_{-1})\backslash\{p\}$ be hyperbolic space with a point removed. 
This has constant scalar curvature equal to $-n(n-1)$ and can be written as a warped product 
$$\left(S^{n-1}\times(0,\infty), \sinh^2(t)g_1+ dt^2\right).$$
Since $\sinh(t)$ is strictly $\log$-concave we see that for $0<\ell_-<\ell_+<\infty$ the warped product
$$\left(S^{n-1}\times[\ell_-,\ell_+],\sinh^2(t)g_1+ dt^2\right)$$
is a model space.

\begin{Bemerkung}
Let $(X^n,g)$ be a Riemannian spin band and $(M,g_\varphi)$ one of the above model spaces. Let $\Phi:X\to M$ be a smooth 1-Lipschitz band map with degree non zero.
In \cite{CeZei21}*{Corollary 10.4} Cecchini and Zeidler prove the following: If $\Sc(X,g)\geq \Sc(M,g)$ and $H(\p_\pm X,g)\geq H(\p_\pm M,g)$, then $\Phi$ is an isometry.

As is explained in \cite{Gro19}*{Section 5.5} one can recreate similar results using $\mu$-bubbles and a stabilized version of Llarull's theorem \cite{Ll98} in dimension $3\leq n\leq7$ (one does not need to assume that $n$ is odd). However, as rigidity for strictly $\log$-concave warping functions remains problematic in our setting (see Remark \ref{rem:rigidity}) the best result we could obtain at this moment is: If $\Sc(X,g)\geq \Sc(M,g)$ and $H(\p_\pm X,g)\geq H(\p_\pm M,g)$, there is no smooth band map $\Phi:X\to M$ with degree non-zero and $\Lip(\Phi)<1$.
\end{Bemerkung}

\section{$\mu$-bubbles}\label{section:bubbles}

We briefly introduce the most important definitions and results (see \cite{ChLi20}*{Section 3}, \cite{Gro19}*{Section 5.1} and \cite{Zhu20}*{Section 2}) concerning $\mu$-bubbles.
As a good reference for the theory of Caccioppoli sets, which will be used freely throughout the rest of this article, we recommend \cite{Giu84}*{Chapter 1}.

\subsection{Basics}
Let $(X,g)$ be an oriented Riemannian band and $h$ be a smooth function on $X$.
Denote by $\mathcal{C}(X)$ the set of all Caccioppoli sets in $X$ which contain an open neighborhood of $\p_-X$ and are disjoint from $\p_+X$.
For $\hat{\Omega}\in\mathcal{C}(X)$ consider the functional
\begin{displaymath}
\mathcal{A}_h(\hat{\Omega})=\Ha^{n-1}(\p^*\hat{\Omega}\cap\mathring{X})-\int_{\hat{\Omega}}hd\Ha^n,
\end{displaymath}
where $\p^*\hat{\Omega}$ is the reduced boundary \cite{Giu84}*{Chapter 3,4} of $\hat{\Omega}$. We denote 
$$\mathcal{I}:=\inf\{\mathcal{A}_h(\hat{\Omega})\bigr| \hat{\Omega}\in\mathcal{C}(X)\}$$
and call a Caccioppoli set $\Omega\in\mathcal{C}(X)$ a $\mu$\emph{-bubble} if $\mathcal{A}_h(\Omega)=\mathcal{I}$ ie $\Omega$ minimizes the $\mathcal{A}_h$-functional among all Caccioppoli sets in $X$, which contain a neighborhood of $\p_-X$ and are disjoint from $\p_+X$.

\begin{Bemerkung}\label{rem:convention2}
To preempt any confusion we remind reader of our mean curvature convention in Remark \ref{rem:convention1}, according to which $H(\p_-X)$ is the trace of the second fundamental form with respect to the inner unit normal field. 
However, if $\hat{\Omega}$ is a smooth Caccioppoli set which contains an open neighborhood of $\p_-X$ and $\hat{\Sigma}$ is a connected component of $\p\hat{\Omega}\cap\mathring{X}$, then the mean curvature $H(\hat{\Sigma})$ is the trace of the second fundamental form with respect to the unit normal field pointing into $\hat{\Omega}$. 
Hence, if $\hat{\Sigma}$ approaches $\p_-X$ then $H(\hat{\Sigma})$ approaches $-H(\p_-X)$ and if $\hat{\Sigma}$ approaches $\p_+X$, then $H(\hat{\Sigma})$ approaches $H(\p_+X)$.
\end{Bemerkung}

\begin{Lemma}[see \cite{Gro19}*{Section 5.1}]\label{lem:exmean}
If $n\leq7$ and $H(\p_\pm X)>\pm h$ on $\p_\pm X$, there is a smooth $\mu$-bubble $\Omega$ ie a smooth Caccioppoli set $\Omega\in\mathcal{C}(X)$, with $\mathcal{A}_h(\Omega)=\mathcal{I}$.
\end{Lemma}

\begin{proof}
We adapt the proofs of \cite{Zhu20}*{Proposition 2.1} and \cite{ChLi20}*{Proposition 12}.
For $t>0$ denote by $\Omega_\pm^t$ the $t$-neighborhoods of $\p_\pm X$. Since $\p_\pm X$ is smooth $\Omega_\pm^t$ has a foliation by smooth equidistant hypersurfaces $\Sigma^{s\leq t}_{\pm}$ for $t$ small enough.
Denote by $\nu^s_\pm$ the unit normal vector field to $\Sigma^s_{\pm}$ pointing in the direction of $\p_+X$ and by $H(\Sigma^s_\pm)$ the trace of the second fundamental form of $\Sigma^s_\pm$ with respect to $-\nu^s_\pm$. 
By possibly making $t$ even smaller we can guarantee
$$\dive(\nu_-^s)=H(\Sigma_-^s)< h(x) \text{ on $\Omega_-^t$ and }\dive(\nu_+^s)=H(\Sigma_+^s)> h(x)\text{ on $\Omega_+^t$}.$$
Let $\hat{\Omega}$ be any Caccioppoli set with $\p_-X\subset\hat{\Omega}$ and $\p_+X\cap\hat{\Omega}=\emptyset$.
We want to see the following: if we add $\Omega_-^t$ to $\hat{\Omega}$ or substract $\Omega_+^t$ from $\hat{\Omega}$ we do not increase the value of $\mathcal{A}_h$.
\begin{multline*}
\mathcal{A}_h((\hat{\Omega}\cup\Omega_-^t)\backslash\Omega_+^t)-\mathcal{A}_h(\hat{\Omega})= \Ha^{n-1}(\p\Omega_-^t\backslash\hat{\Omega})-\Ha^{n-1}(\p^*\hat{\Omega}\cap\Omega_-^t)+\Ha^{n-1}(\p\Omega_+^t\cap\hat{\Omega})\\
-\Ha(\p^*\hat{\Omega}\cap\Omega_+^t)
-\int_{\Omega_-^t\backslash\hat{\Omega}}hd\Ha^n+\int_{\Omega_+^t\cap\hat{\Omega}}hd\Ha^n.
\end{multline*}
The divergence theorem and our assumption on $h$ implies 
\begin{displaymath}
\int_{\Omega_-^t\backslash\hat{\Omega}}hd\Ha^n> \int_{\Omega_-^t\backslash\hat{\Omega}}\dive{\nu_-^s}=\int_{\p^*(\Omega_-^t\backslash\hat{\Omega})}\langle\nu_-^s,\nu\rangle d\Ha^{n-1}\geq\Ha^{n-1}(\p\Omega_-^t\backslash\hat{\Omega})-\Ha^{n-1}(\p^*\hat{\Omega}\cap\Omega_-^t)
\end{displaymath}
and
\begin{displaymath}
\int_{\Omega_+^t\cap\hat{\Omega}}hd\Ha^n<\int_{\Omega_+^t\cap\hat{\Omega}}\dive{\nu_+^s}=\int_{\p^*(\Omega_+^t\cap\hat{\Omega})}\langle\nu_+^s,\nu\rangle d\Ha^{n-1}\leq-\Ha^{n-1}(\p\Omega_+^t\cap\hat{\Omega})
+\Ha^{n-1}(\p^*\hat{\Omega}\cap\Omega_+^t).
\end{displaymath}
We conclude that 
\begin{equation}\label{eq:min}
\mathcal{A}_h((\hat{\Omega}\cup\Omega_-^t)\backslash\Omega_+^t)-\mathcal{A}_h(\hat{\Omega})<0,
\end{equation}
which implies that it is enough to search for a minimizer among all Caccioppoli sets in $X$ with $\Omega_-^t\subset\hat{\Omega}$ and $\Omega_+^t\cap\hat{\Omega}=\emptyset$. If $C$ is a constant such that $|h|<C$ on $X$, then for any such Caccioppoli set we have $\mathcal{A}_h(\hat{\Omega})>-C\Ha^{n}(X)>-\infty$.
We choose a minimizing sequence $\hat{\Omega}_k$. By compactness for Caccioppoli sets (compare \cite{Giu84}*{Theorems 1.19 \& 1.20}) $\hat{\Omega}_k$ subconverges to a minimizing Caccioppoli set $\Omega$ which contains an open neighborhood of $\p_-X$ and is disjoint from $\p_+X$. Smoothness of $\Omega$ follows from the regularity theorem \cite{ZZ18}*{Theorem 2.2}.
\end{proof}

If $\hat{\Omega}\in\mathcal{C}(X)$ is smooth and $\hat{\Sigma}$ is a connected component of $\p\hat{\Omega}\backslash\p_-X$, we denote by $\nu$ the outwards pointing unit normal vector field, by $A$ the second fundamental form with respect to $-\nu$ and by $H$ the trace of $A$.

\begin{Lemma}[First variation formula]\label{lem:var1nonwarped}
For any smooth function $\psi$ on $\hat{\Sigma}$ let $V_\psi$ be a vector field on $X$, which vanishes outside a small neighborhood of $\hat{\Sigma}$ and agrees with $\psi\nu$ on $\hat{\Sigma}$. If we denote by $\Phi_t$ the flow generated by $V_\psi$, then
\begin{equation}\label{eq:var1}
\frac{d}{dt}\Bigr|_{t=0}\mathcal{A}_h(\Phi_t(\hat{\Omega}))=\int_{\hat{\Sigma}} (H-h)\psi d\Ha^{n-1}.
\end{equation}
\end{Lemma}

\begin{Lemma}[Second variation formula]\label{lem:var2nonwarped}
For any smooth function $\psi$ on $\hat{\Sigma}$ let $V_\psi$ be a vector field on $X$, which vanishes outside a small neighborhood of $\hat{\Sigma}$ and agrees with $\psi\nu$ on $\hat{\Sigma}$. If we denote by $\Phi_t$ the flow generated by $V_\psi$, then
\begin{displaymath}
\frac{d^2}{dt^2}\Bigr|_{t=0}\mathcal{A}_h(\Phi_t(\hat{\Omega}))= \int_{\hat{\Sigma}} |\nabla_{\hat{\Sigma}}\psi|^2+(H^2-Ric(\nu,\nu)-|A|^2-Hh-g(\nabla_Xh,\nu))\psi^2,
\end{displaymath}
which is equal to
\begin{equation}\label{eq:var2nonwarped}
\int_\Sigma|\nabla_\Sigma\psi|^2-\frac{1}{2}(\Sc(X,g)-\Sc(\hat{\Sigma},g)-H^2+|A|^2)\psi^2-(Hh+g(\nabla_X(h),\nu))\psi^2
\end{equation}
\end{Lemma}

\begin{proof}
We differentiate the first variation employing the following Leibniz rule: If $f$ is a smooth function on $X$, then 
\begin{displaymath}
\frac{d}{dt}\Bigr|_{t=0}\int_{\hat{\Sigma}_t}f\psi d\Ha^{n-1}= \int_{\hat{\Sigma}} (Hf+g(\nabla_Xf,\nu))\psi^2 d\Ha^{n-1}.
\end{displaymath}
Furthermore we use the formula
\begin{displaymath}
\int_{\hat{\Sigma}}g(\nabla_XH_{\hat{\Sigma}_t},\nu) \psi^2 d\Ha^{n-1}=\int_{\hat{\Sigma}} |\nabla_{\hat{\Sigma}}\psi|^2-(Ric(\nu,\nu)+|A|^2)\psi^2,
\end{displaymath}
and the standard trick to rewrite $Ric(\nu,\nu)$ from \cite{SY79}*{p. 165}.
\end{proof}

If $\Omega$ is the $\mu$-bubble we get from Lemma \ref{lem:exmean}, then the mean curvature $H$ of $\Sigma$ is equal to $h$ by Lemma \ref{lem:var1nonwarped} and by stability and Lemma \ref{lem:var2nonwarped}, we see that 
\begin{displaymath}
\begin{split}
0&\leq\int_\Sigma|\nabla_\Sigma\psi|^2-\frac{1}{2}(\Sc(X,g)-\Sc(\Sigma,g)-H^2+|A|^2)\psi^2-(Hh+g(\nabla_Xh,\nu))\psi^2\\
&=\int_\Sigma|\nabla_\Sigma\psi|^2-\frac{1}{2}(\Sc(X,g)-\Sc(\Sigma,g)+H^2+|A|^2)\psi^2-g(\nabla_Xh,\nu)\psi^2\\
&\leq\int_\Sigma|\nabla_\Sigma\psi|^2-\frac{1}{2}(\Sc(X,g)-\Sc(\Sigma,g)+\frac{n}{n-1}h^2+2g(\nabla_Xh,\nu))\psi^2,
\end{split}
\end{displaymath}
where we used $|A|^2\geq\frac{H^2}{n-1}$ for the last inequality. By rearranging terms, we conclude
\begin{equation}\label{eq:key}
\int_\Sigma |\nabla_\Sigma\psi|^2+\frac{1}{2}\Sc(\Sigma,g)\psi^2 d\Ha^{n-1}\geq\int_\Sigma\frac{1}{2}(\Sc(X,g) +\frac{n}{n-1}h^2+2g(\nabla_Xh,\nu))\psi^2d\Ha^{n-1}.
\end{equation}

We are now ready to prove Proposition \ref{prop:comparison1}:

\begin{proof}[Proof of Proposition\ref{prop:comparison1}]
By assumption there is a structural band map $\phi:X\to[a,b]$. Thus $h=h_\varphi\circ\phi$ is a smooth function on $X$ and by assumption $H(\p_\pm X)>\pm h$. 
Since $n\leq7$ Lemma \ref{lem:exmean} yields a smooth minimizer $\Omega$ for the $\mathcal{A}_h$-functional, which contains a neighborhood of $\p_-X$ and is disjoint from $\p_+X$. 
Hence $\Sigma=\p\Omega$ separates $\p_-X$ and $\p_+X$ in $X$. 
Furthermore by \eqref{eq:key}
$$\int_\Sigma |\nabla_\Sigma\psi|^2+\frac{1}{2}\Sc(\Sigma,g)\psi^2 d\Ha^{n-1}\geq\int_\Sigma\frac{1}{2}(\Sc(X,g) +\frac{n}{n-1}h^2+2g(\nabla_Xh,\nu))\psi^2d\Ha^{n-1}$$
for any $\psi\in C^\infty(\Sigma)$. 
Since $\phi$ is structural  we have
$$\int_\Sigma |\nabla_\Sigma\psi|^2+\frac{1}{2}\Sc(\Sigma,g)\psi^2 d\Ha^{n-1}\geq\int_\Sigma\frac{1}{2\varphi^2(\phi)}Sc(N,g_N)\psi^2$$
and hence
$$-\Delta_{\Sigma}+\frac{1}{2}\Sc(\Sigma,g)\geq\frac{1}{2\varphi^2(\phi)}Sc(N,g_N).$$
\end{proof}

\begin{Bemerkung}\label{rem:sharp1}
Let $(X,g)$ be an oriented Riemannian band, $(M,g_\varphi)$ a warped product over $(N,g_N)$ and $\phi:X\to[a,b]$ a smooth band map. If $\varphi$ is strictly $\log$-concave, $\Sc(X,g)(x)\geq\Sc(M,g_\varphi(p_0,\phi(x))$ and $\phi$ has $\Lip(\phi)<1$, then $\phi$ is structural by Lemma \ref{lem:logconcavestructure}. 
As was observed in Remark \ref{rem:addition} we even get strict inequality in \eqref{eq:structural} in this case as $g(\nabla_Xh,\nu)>h'_\varphi$.
Hence the argument above yields
$$-\Delta_{\Sigma}+\frac{1}{2}\Sc(\Sigma,g)>\frac{1}{2\varphi^2(\phi)}Sc(N,g_N).$$
\end{Bemerkung}

\subsection{Constant Mean Curvature Surfaces}\label{section:CMC} If $\Omega$ is a smooth minimizer for the $\mathcal{A}_h$ functional, then $\Sigma=\p\Omega\cap\mathring{X}$ is often called a prescribed mean curvature (or short PMC) surface in the literature (see for example \cite{ZZ18}).
This terminology is based on the observation that $H(\Sigma)=h\bigr|_\Sigma$ by the first variation formula. In the following we assume $h$ to be a constant function. 
In this case $\Sigma$ is called a constant mean curvature (or short CMC) surface and our main goal is to understand what happens if we relax the strict boundary condition $H(\p_\pm X)>\pm h$ to $H(\p_\pm X)\geq\pm h$ in Lemma \ref{lem:exmean}.

In the proof of Lemma \ref{lem:exmean} the assumption $H(\p_\pm X)>\pm h$ was used to show that there is a minimizing sequence $\hat{\Omega}_k$ in $\mathcal{C}(X)$ which converges to a Caccioppoli set $\Omega\in\mathcal{C}(X)$. 
This fails if we relax to $H(\p_\pm X)\geq\pm h$, since it might happen that the limit $\Omega$ of any minimizing sequence $\hat{\Omega}_k$ in $\mathcal{C}(X)$ contains points of $\p_+X$ or does not contain a neighborhood of $\p_-X$ any more. However, for $h$ constant, we can use a strong maximum principle to address this issue.

To make this precise we slightly change our set up. Let $(X,g)$ be an oriented Riemannian band and $h$ be constant function on $X$.
Without loss of generality we can assume $h$ to be nonnegative (otherwise we just change the roles of $\p_-X$ and $\p_+X$).
For some $\delta>0$ we glue on collars $\p_-X\times(-\delta,0]$ and $\p_+X\times[0,\delta)$ on both sides of $X$ and extend the metric $g$ smoothly to produce a Riemannian manifold $(X_\delta,g_\delta)$. 
This can be done in such a way that $\vol(X_\delta,g_\delta)<\vol(X,g)+\delta$.\\
Let $\mathcal{C}(X_\delta)$ be the set of all Caccioppoli sets in $X_\delta$, which contain $\p_-X\times(-\delta,0]$ and are disjoint from $\p_+X\times(0,\delta)$.
We replace $\Ha^{n-1}(\p^*\hat{\Omega}\cap\mathring{X})$ by $\Ha^{n-1}(\p^*\hat{\Omega})$ in the $\mathcal{A}_h$-functional and define $\mathcal{I}_\delta:=\inf\{\mathcal{A}_h(\hat{\Omega})\bigr| \hat{\Omega}\in\mathcal{C}(X_\delta)\}$.

\begin{Proposition}\label{prop:strongmax}
Let $h\geq0$ be constant and $n\leq7$. If $H(\p_\pm X)\geq\pm h$ and $\Omega\in\mathcal{C}(X_\delta)$ is a minimizer ie $\mathcal{A}_h(\Omega)=\mathcal{I}_\delta$, then any connected component of $\p\Omega$ is either contained in $\mathring{X}$ or agrees with a connected component of $\p_-X$ resp. $\p_+X$.
\end{Proposition}

\begin{proof}
Before we start we invoke a regularity result from \cite{HI01}*{Theorem 1.3} which in turn is based on \cite{Ta84}. It implies that $\p\Omega$ is a $C^{1,\frac{1}{2}}$ submanifold of $X$. Hence $\p^*\Omega=\p\Omega$ and $\Omega$ has a $C^{0,\frac{1}{2}}$ outer unit normal vector field $\nu$.
We assume that $\p\Omega$ is connected. Otherwise we treat each connected component separately.

For $h=0$ the statement follows directly from the strong maximum principle for minimal hypersurfaces (see \cite{Wh10}*{Theorem 4} or \cite{SW89}) applied to $\p^*\Omega$.
If $h>0$, we can apply the strong maximum principle for hypersurfaces of bounded variation \cite{Wh10}*{Theorem 7} to see that $\p\Omega$ can only touch a component of $\p_+X$ if they agree. Since these result hold for general varifolds we did not yet use the a priori regularity of $\p\Omega$.

To see that $\p\Omega$ can only touch a component of $\p_-X$ if they agree, we follow the standard recipe for proving a strong maximum principle as it is explained in \cite{SW89} and \cite{Wh10}.
First we assume that there is a point $p\in\p_-X$ with $H(\p_-X,g)(p)>\eta>-h$.
By \cite{Wh10}*{Theorem 2} there is a compactly supported vector field $V$ on $X_\delta$ such that $V(p)$ is a nonzero normal to $\p_-X$ and
$$\int_{\p\Omega}\dive_{\p\Omega}Vd\Ha^{n-1}+\int_{\p\Omega}\eta|V|d\Ha^{n-1}\leq0.$$
If we choose the support of $V$ small enough and assume that $\p\Omega$ touches $\p_-X$ at $p$ ie $p\in\p\Omega$ and the normal vectors coincide, then
$$0\leq\int_{\p\Omega}\dive_{\p\Omega}Vd\Ha^{n-1}-\int_{\p\Omega}hg(V,\nu)d\Ha^{n-1}<\int_{\p\Omega}\dive_{\p\Omega}Vd\Ha^{n-1}+\int_{\p^*\Omega}\eta|V|d\Ha^{n-1}\leq0,$$
which is a contradiction. Hence $p\notin\p\Omega$.

Now let $p\in\p_-X$ be arbitrary and assume that $\p\Omega$ touches $\p_-X$ at $p$ but does not coincide with $\p_-X$ in any neighborhood of $p$.
We proceed as in \cite{SW89}*{Step 1, p.687} (see also the comments at the end of \cite{SW89}) and use the implicit function theorem (see \cite{Wh87}*{Appendix}) to foliate a neighborhood of $p$ in $X_\delta$ by hypersurfaces with controlled mean curvature.
By the Hopf maximum principle there is such a hypersurface with mean curvature strictly bigger than $-h$, which lies on the same side of $\p\Omega$ as $\p_-X$ and touches $\p\Omega$ but they do not coincide. This leads to a contradiction as we have seen before. Hence $\p\Omega$ and $\p_-X$ agree on a neighborhood of $p$. Using a standard open-closed-connected argument we conclude that $\p\Omega=\p_-X$ and $H(\p_-X,g)=-h$.
\end{proof}

\begin{Lemma}\label{lem:exconst}
If $h$ is constant, $n\leq7$ and $H(\p_\pm X)\geq\pm h$ on $\p_\pm X$, there is a smooth Caccioppoli set $\Omega\in\mathcal{C}(X_\delta)$, with $\mathcal{A}_h(\Omega)=\mathcal{I}_\delta$.
\end{Lemma}

\begin{proof}
Since $\vol(X_\delta,g_\delta)<\vol(X,g)+\delta$ the $\mathcal{A}_h$-functional is bounded from below on $\mathcal{C}(X_\delta)$.
By compactness for Caccioppoli sets there is a minimizer $\Omega\in\mathcal{C}(X_\delta)$.
By Proposition \ref{prop:strongmax} any connected component of $\p\Omega$ is either contained in $\mathring{X}$ and hence smooth by the regularity theorem \cite{ZZ18}*{Theorem 2.2} or agrees with a connected component of $\p_-X$ resp. $\p_+X$.
\end{proof}

Let $\Omega$ be the minimizer from Lemma \ref{lem:exconst} and $\Sigma=\p\Omega$.
It is important to note that $\Omega$ is only stationary and stable for variations which preserve $X$ which is the case if and only if the variation vector field has nonnegative scalar product with the interior normal vector fields to $\p_\pm X$. 
Let $\Sigma_0\subset\Sigma$ be a connected component. If $\Sigma_0\subset\mathring{X}$ all variation vector fields $V_\psi$ are admissible and we conclude $H(\Sigma_0)=h\bigr|_{\Sigma_0}$ by the first variation formula. By the second variation formula and stability \eqref{eq:key} holds for all $\psi\in C^\infty(\Sigma_0)$.

If $\Sigma_0$ agrees with a component of $\p_- X$ (the case $\Sigma_0\subset\p_+X$ follows in analogous fashion), we only consider variation vector fields $V_\psi$ with $\psi\geq0$. By the first variation formula
$$\int_{\hat{\Sigma}} (H-h)\psi d\Ha^{n-1}\geq0$$
for all nonnegative $\psi\in C^\infty(\Sigma_0)$. Since $(H-h)$ is negative on $\p_-X$ by assumption (remember Remark \ref{rem:convention2} ie $H(\Sigma_0)=-H(\p_-X)$) this implies $H(\Sigma_0)=h\bigr|_{\Sigma_0}$.

By stability and the second variation formula \eqref{eq:key} holds for all $\psi\in C^\infty(\Sigma_0)$ with $\psi\geq0$. Since the first eigenfunction of the operator
$$-\Delta_{\Sigma_0}+\frac{1}{2}\Sc(\Sigma_0,g)-\frac{1}{2}(\Sc(X,g) +\frac{n}{n-1}h^2+2g(\nabla_Xh,\nu))$$
does not change sign (follows from elliptic regularity and the Hopf maximum principle), this implies that the operator is nonnegative.

\begin{Bemerkung}\label{rem:sharp2}
Using Lemma \ref{lem:exconst}, together with the argument above, we can prove Proposition \ref{prop:comparison1} for constant $h_\varphi$ with the weakened boundary condition $H(\p_\pm X)\geq\pm h$.
\end{Bemerkung}

\subsection{Warped $\mu$-Bubbles} The following version of $\mu$-bubbles was introduced in \cite{ChLi20}*{Section 3}.
The results of this subsection will be used exclusively in Section \ref{section:last}.
Let $(X,g)$ be an oriented Riemannian band. Let $u>0$ be a smooth function on $X$ and $h$ be a smooth function on $\mathring{X}$ respectively $X$.
We fix a Caccioppoli set $\Omega_0$ with smooth boundary, which contains an open neighborhood of $\p_-X$ and is disjoint from $\p_+X$.
Hence all components of $\p\Omega_0$, which are not part of $\p_-X$ are contained in $\mathring{X}$. We consider
\begin{displaymath}
\mathcal{A}^{u}_h(\hat{\Omega})=\int_{\p^*\hat{\Omega}}ud\Ha^{n-1}-\int_X(\chi_{\hat{\Omega}}-\chi_{\Omega_0})hud\Ha^n
\end{displaymath}
for all Caccioppoli sets $\hat{\Omega}$ with $\hat{\Omega}\Delta\Omega_0$ contained in the interior of $X$ (this implies in particular, that $\hat{\Omega}$ contains an open neighborhood of $\p_-X$ and is disjoint from $\p_+X$).
A Caccioppoli set, which is minimizing $\mathcal{A}^{u}_h$ in this class, is called a \emph{warped $\mu$-bubble}. The following existence and regularity result is \cite{Zhu20}*{Proposition 2.1} and \cite{ChLi20}*{Proposition 12}.

\begin{Lemma}\label{lem:ex}
If $n\leq7$ and $h(x)\rightarrow\pm\infty$ as $x\rightarrow\p_\mp X$, there exists a smooth minimizer $\Omega$ for $\mathcal{A}^{u}_h$, such that  $\Omega\Delta\Omega_0$ is contained in the interior of $X$.
\end{Lemma}

\begin{Lemma}[Warped first variation formula]\label{lem:var1}
For any smooth function $\psi$ on $\hat{\Sigma}$ let $V_\psi$ be a vector field on $X$, which vanishes outside a small neighborhood of $\hat{\Sigma}$ and agrees with $\psi\nu$ on $\hat{\Sigma}$. If we denote by $\Phi_t$ the flow generated by $V_\psi$, then
\begin{equation}\label{eq:var1}
\frac{d}{dt}\Bigr|_{t=0}\mathcal{A}^{u}_h(\Phi_t(\hat{\Omega}))=\int_{\hat{\Sigma}} (Hu+g(\nabla_Xu,\nu)-hu)\psi d\Ha^{n-1}.
\end{equation}
\end{Lemma}

\begin{Lemma}[Warped second variation formula]\label{lem:var2}
For any smooth function $\psi$ on $\hat{\Sigma}$ let $V_\psi$ be a vector field on $X$, which vanishes outside a small neighborhood of $\hat{\Sigma}$ and agrees with $\psi\nu$ on $\hat{\Sigma}$.
If we denote by $\Phi_t$ the flow generated by $V_\psi$, then
\begin{multline*}
\frac{d^2}{dt^2}\Bigr|_{t=0}\mathcal{A}^{u}_h(\Phi_t(\hat{\Omega}))= \int_{\hat{\Sigma}} |\nabla_{\hat{\Sigma}}\psi|^2u+(H^2-Ric(\nu,\nu)-|A|^2)\psi^2u +\\+(2Hg(\nabla_Xu,\nu)+\frac{d^2u}{d\nu^2}-Hhu-g(\nabla_X(hu),\nu)\psi^2,
\end{multline*}
which is equal to
\begin{equation}\label{eq:var2}
\int_{\hat{\Sigma}}|\nabla_{\hat{\Sigma}}\psi|^2u-\frac{1}{2}(\Sc(X,g)-\Sc(\hat{\Sigma},g)-H^2+|A|^2)\psi^2u+(2Hg(\nabla_Xu,\nu)+\frac{d^2u}{d\nu^2}-Hhu-g(\nabla_X(hu),\nu))\psi^2.
\end{equation}
\end{Lemma}

\section{Proof of the Main Theorem}\label{section:comparisonresults}

In this section we prove parts \eqref{itm:A} and \eqref{itm:B} of our Main Theorem. We even establish the following slightly more general statements using the techniques from Section \ref{section:warpedproducts} and \ref{section:bubbles}.

\begin{Satz}\label{thm:comparisongeneralA}
Let $n\leq7$ and $(X^n,g)$ be an oriented Riemannian band with the property that no hypersurface $\Sigma$ which separates $\p_-X$ and $\p_+X$ has $-\Delta_\Sigma+\frac{1}{2}\Sc(\Sigma,g)>0$.
Let $(M,g_\varphi)$ be a model space over a scalar flat base with warping function $\varphi:[a,b]\to\R_+$. If $\varphi$ is strictly $\log$-concave,
\begin{itemize}
\item[$\triangleright$] $\Sc(X,g)\geq\Sc(M,g_\varphi)$,
\item[$\triangleright$] $H(\p_\pm X,g)\geq H(\p_\pm M,g_\varphi)$,
\end{itemize}
then $\wid(X,g)\leq\wid(M,g_\varphi)$.
\end{Satz}

\begin{proof}
If we assume for a contradiction that $\wid(X,g)>\wid(M,g_\varphi)=b-a$, there is a small $\e>0$, such that $\wid(X,g)>b-a+2\e$.
Let $(M_\e, g^\e_\varphi)$ be the model space $$\left(N\times[a-\e,b+\e],\varphi^2g_N+dt^2\right).$$
We compare $(X,g)$ and $(M_\e, g^\e_\varphi)$.
According to Lemma \ref{lem:3condstructure} there is a structural map $\phi:(X,g)\rightarrow[a-\e,b+\e]$ with $\Lip(\phi)<1$. Since $\varphi$ is strictly $\log$-concave and $H(\p_\pm X,g)\geq H(\p_\pm M,g_\varphi)$ we have $H(\p_\pm X,g)>H(M_\e, g^\e_\varphi)$.
Proposition \ref{prop:comparison1}, together with Remark \ref{rem:addition} and Remark \ref{rem:sharp1}, implies the existence of a hypersurface $\Sigma$, which separates $\p_-X$ and $\p_+X$ and has $-\Delta_\Sigma+\frac{1}{2}\Sc(\Sigma,g)>0$. This is a contradiction.
\end{proof}

Theorem \ref{thm:comparisongeneralA} implies part \eqref{itm:A} of Theorem \ref{MainTheorem} with the help of the following classical result of Kazdan-Warner \cite{KW75} and Schoen-Yau \cite{SY79}:

\begin{Lemma}\label{lem:conf}
Let $(\Sigma^{n\geq2},g)$ be a closed connected oriented manifold. If $-\Delta_\Sigma+\frac{1}{2}\Sc(\Sigma,g)$ is positive, then $\Sigma$ admits a metric with positive scalar curvature.
\end{Lemma}

\begin{proof}
The proof is standard so we only recall  the main ideas. Since the operator is positive
$$
\int_\Sigma-\psi\Delta_g\psi+\frac{1}{2}\Sc(\Sigma,g)\psi^2>0
$$
for all $\psi\in C^2(\Sigma)$. If $n=2$ we choose $\psi\equiv1$ and use Gauss-Bonnet to see that 
$$
0<\int_\Sigma \frac{1}{2}\Sc(\Sigma,g)=2\pi\chi(\Sigma).
$$
It follows that $\Sigma$ is a 2-sphere and hence admits a metric with positive scalar curvature.\\
If $n\geq3$, we consider the conformal Laplacian $L_g=-\Delta_g+\frac{n-2}{4(n-1)}\Sc(\Sigma,g)$. It is easy to see that this operator is positive as well.
Hence the first eigenvalue $\lambda_1(L_g)$ is positive.
It follows from elliptic regularity and the strong maximum principle that the first eigenfunction $u\in C^\infty(\Sigma)$ can be chosen positive.\\
We then use this function for a conformal change of metric ie $\hat{g}=u^{\frac{4}{n-2}}g$. We conclude
$$
\Sc(\Sigma,{\hat{g}})=u^{-\frac{n+2}{n-2}}\frac{4(n-1)}{n-2} L_gu>0
$$
using the standard formula for scalar curvature under a conformal change of metric.
\end{proof}

\begin{Satz}\label{thm:comparisongeneralB}
Let $n\leq7$ and $(X^n,g)$ be an oriented Riemannian band with the property that no hypersurface $\Sigma$ which separates $\p_-X$ and $\p_+X$ admits a metric with positive scalar curvature.
Let $(M,g_\varphi)$ be a model space over a scalar flat base with warping function $\varphi:[a,b]\to\R_+$. If $\varphi$ is $\log$-constant, 
\begin{itemize}
\item[$\triangleright$] $\Sc(X,g)\geq\Sc(M,g_\varphi)$,
\item[$\triangleright$] $H(\p_\pm X,g)\geq H(\p_\pm M,g_\varphi)$,
\end{itemize}
then $(X,g)$ is isometric to a warped product
$$\left(\hat{N}\times[c,d], \varphi^2g_{\hat{N}}+dt^2 \right),$$
where $(\hat{N},g_{\hat{N}})$ is a closed scalar flat Riemannian manifold.
\end{Satz}

\begin{proof}
Let $\phi:X\to[a,b]$ be a band map.
According to Lemma \ref{lem:constantstructure} $\phi$ is structural.
Following the proof Proposition \ref{prop:comparison1}, together with Remark \ref{rem:sharp2}, we see that there is a hypersurface $\Sigma$, which separates $\p_-X$ and $\p_+X$ and has 
$$-\Delta_\Sigma+\frac{1}{2}\Sc(\Sigma,g)\geq\frac{1}{2}\left(\Sc(X,g) +h^2+|A|^2\right)\geq0,$$
where $h=h_\varphi\circ\phi$.

By our assumption on $X$, there is a connected component $\Sigma_0\subset\Sigma$ which does not admit a metric of positive scalar curvature. 
Considering Lemma \ref{lem:conf}, we conclude that the first eigenvalue of $-\Delta_{\Sigma_0}+\frac{1}{2}\Sc(\Sigma_0,g)$ is equal to zero.

If $w_1$ is the corresponding positive first eigenfunction, then 
$$\int_{\Sigma_0} \frac{1}{2}\left(\Sc(X,g) +h^2+|A|^2\right)w_1^2=0.$$
Consequently $\Sc(X,g)+h^2+|A|^2=0$ which is equivalent to $-h^2-|A|^2=\Sc(X,g)$. 

On the other hand $\Sc(X,g)\geq\Sc(M,g_\varphi)=-\frac{n}{n-1}h_\varphi^2=-\frac{n}{n-1}h^2$ by \eqref{eq:warpedODE}.
Since $|A|^2\geq\frac{H^2}{n-1}=\frac{h^2}{n-1}$ we conclude that $\Sc(X,g)=\Sc(M,g_\varphi)$ along $\Sigma_0$ and $|A|^2=\frac{H^2}{n-1}=\frac{h^2}{n-1}$.

Regarding $\Sc(\Sigma_0,g)$ we distinguish three cases: If $n=2$, the term $\Sc(\Sigma_0,g)$ does not appear.
If $n=3$, we choose $\psi\equiv1$ in
$$\int_{\Sigma_0} |\nabla_{\Sigma_0}\psi|^2+\frac{1}{2}\Sc(\Sigma_0,g)\psi^2\geq0.$$
By Gauss-Bonnet $\Sigma_0$ is a torus and $\Sc(\Sigma_0,g)=0$.

If $n>3$ we proceed as in \cite[p. 166]{SY79} and consider the first positive eigenfunction $w_2\in C^{\infty}(\Sigma_0)$ corresponding to the first eigenvalue $\lambda_0$ of the conformal Laplacian
$$L_g=-\Delta_{\Sigma_0}+\frac{(n-3)}{4(n-2)}\Sc(\Sigma_0,g).$$
By Lemma \ref{lem:conf} we see that $\lambda_0\leq0$. Hence
\begin{displaymath}
\frac{2(n-2)}{n-3}\int_{\Sigma_0}|\nabla_{\Sigma_0}w_2|^2=-\int_{\Sigma_0}\frac{1}{2}\Sc(\Sigma_0,g)w_2^2+\frac{2\lambda_0(n-2)}{n-3}\int_{\Sigma_0} w_2^2\leq\int_{\Sigma_0}|\nabla_{\Sigma_0}w_2|^2.
\end{displaymath}
Since $\frac{2(n-2)}{n-3}>1$ we see that $\lambda_0=0$ and $w_2$ is a constant function.
Consequently $\Sc(\Sigma_0,g)$ is constant as well. Since the first eigenvalue of $-\Delta_{\Sigma_0}+\frac{1}{2}\Sc(\Sigma_0,g)$ is equal to zero we conclude $\Sc(\Sigma_0,g)=0$.\\

It follows that the Jacobi operator associated to $\Sigma_0$ is 
$$-\Delta_{\Sigma_0}-({\rm Ric}(\nu,\nu)+|A|^2)=-\Delta_{\Sigma_0}-\frac{1}{2}(\Sc(X,g)-\Sc(\Sigma_0,g)+H^2+|A|^2)=-\Delta_{\Sigma_0}.$$
We follow the proof of \cite[Theorem 2.3]{ACG08} respectively \cite[Lemma 3.4]{Zhu20} and use the implicit function theorem to show that there is a foliation $\{\Sigma_s\}_{-\delta<s<\delta}$ around $\Sigma_0$ such that 
\begin{itemize}
\item[$\triangleright$] each $\Sigma_s$ is a graph over $\Sigma_0$ with graph function $u_s$ along the outward unit normal field $\nu$ with 
\begin{equation} \label{eq:fol}
\frac{d}{d s}\bigr|_{s=0}u_s=1\text{ and } \int_{\Sigma_0}u_sd\Ha^{n-1}=s;
\end{equation}\label{eq:graph}
\item[$\triangleright$] $H_s=H(\Sigma_s)-h$ is a constant function on $\Sigma_s$.
\end{itemize}

For $s\in[0,\delta)$ let $\Omega_s$ be the union of $\Omega$ and the region bounded by $\Sigma_0$ and $\Sigma_s$.
Since $\Omega$ minimizes the $\mathcal{A}_h$ functional, there is a $0<\delta'\leq\delta$ such that $H_s-h\geq0$ for all $s\in[0,\delta')$.
There are two possibilities: either $H_s-h>0$ for some $s\in[0,\delta')$ or $H_s-h=0$ for all $s\in[0,\delta')$.

In the first case we choose a constant $0\leq h<\hat{h}<H_s$ and consider the $\mathcal{A}_{\hat{h}}$-functional on the Riemannian band $\hat{X}$ bounded by $\Sigma_0$ and $\Sigma_s$. 
By Lemma \ref{lem:exmean} there is a smooth hypersurface $\hat{\Sigma}$ which separates $\Sigma_0$ and $\Sigma_s$ with $H(\hat{\Sigma})=\hat{h}$ and by stability and the second variation formula we see
$$-\Delta_{\hat{\Sigma}}+\frac{1}{2}\Sc(\hat{\Sigma},g)\geq\frac{1}{2}\left(\Sc(X,g)+\frac{n}{n-1}\hat{h}^2\right)>0.$$
We replace the component $\Sigma_0\subset\Sigma$ by $\hat{\Sigma}$ and restart our argument from the beginning. 
Note that we have reduced the number of components of $\Sigma$ we need to consider by one since $\hat{\Sigma}$ admits a metric with positive scalar curvature by Lemma \ref{lem:conf}.

In case $H_s-h=0$ for all $s\in[0,\delta')$ we show that $\Omega_s$ is a minimizer for the $\mathcal{A}_h$-functional as well:
$$\mathcal{A}_h(\Omega_s)-\mathcal{A}_h(\Omega)=\Ha^{n-1}(\Sigma_s)-\Ha^{n-1}(\Sigma)-\int_{\hat{X}}hd\Ha^{n}=\int_{0}^s\int_{\Sigma_t}f_t(H_t-h) d\Ha^{n-1}dt=0,$$
where $f_t=g(\frac{d}{dt}u_t,\nu_t)$ is the lapse function of $\Sigma_t$ moving along the foliation.

By stability and the second variation formula $-\Delta_{\Sigma_s}+\frac{1}{2}\Sc(\Sigma_s,g)\geq0$. Again, if the inequality is strict, we replace $\Sigma_0$ by $\Sigma_s$ and restart our argument, reducing the number of components $\Sigma$ we have to consider by one in the process.

Hence we assume that for all $s\in[0,\delta')$ the first eigenvalue of $-\Delta_{\Sigma_s}+\frac{1}{2}\Sc(\Sigma_s,g)$ is equal to zero.
We conclude that $\Sc(X,g)=\Sc(M,g_\varphi)$ along $\Sigma_s$ and 
$$|A_s|^2=\frac{H_s^2}{n-1}=\frac{h^2}{n-1}$$
ie all principal curvatures of $\Sigma_s$ are equal to $\frac{h}{n-1}$.
Furthermore $\Sc(\Sigma_s,g)=0$ and the Jacobi operator of $\Sigma_s$ is $-\Delta_{\Sigma_s}$.

With the foliation we can write the metric as $g=g_s+f_s^2ds^2$, where $g_s=g\bigr|_{\Sigma_s}$.
As the lapse function $f_s$ satisfies the Jacobi equation \cite[Equation (1.2)]{HI01}, which reduces to $\Delta_{\Sigma_s}f_s=0$, we see that $f_s$ is constant. By rescaling the $s$-coordinate if necessary we can assume $f_s=1$ and hence $\Sigma_s$ is $s$-equidistant to $\Sigma_0$. 
Since $\Sigma_s$ is umbilic for all $s\in[0,\delta')$ we conclude 
\begin{equation}\label{eq:metricproduct}
g=\exp(2s\frac{h}{n-1})g_0+ds^2
\end{equation}
in $\Sigma_0\times[0,\delta')$. In particular $\{\Sigma_s\}_{0\leq s<\delta'}$ is a geodesic flow starting from $\Sigma_0$.

Since $|A_s|^2=\frac{h^2}{n-1}$ for all $s\in[0,\delta')$ and $\Ha^{n-1}(\Sigma_s)\leq \mathcal{I}+h\vol(X,g)<\infty$, for some $s_k\to \delta'$ the sequence $\Omega_{s_k}$ resp. $\Sigma_{s_k}$ converges to a $\Omega_{\delta'}$ resp. $\Sigma_{\delta'}$. 
Since the $\Omega_{s_k}$ minimize the $\mathcal{A}_h$-functional, so does $\Omega_{\delta'}$.
Hence $\Omega_{\delta'}$ is smooth which implies that $\Sigma_{\delta'}$ does not intersect any component of $\Sigma\backslash\Sigma_0$.

By stability and the second variation formula $-\Delta_{\Sigma_{\delta'}}+\frac{1}{2}\Sc(\Sigma_{\delta'},g)\geq0$.
If the inequality is strict we replace $\Sigma_0$ by $\Sigma_{\delta'}$ and restart our argument.
As we always reduce the number of components of $\Sigma$ we need to consider by one, this process terminates and we can assume without loss of generality that $-\Delta_{\Sigma_i}+\frac{1}{2}\Sc(\Sigma_{i},g)>0$ for every component $\Sigma_i\subset\Sigma\backslash\Sigma_0$.

Hence the first eigenvalue of $-\Delta_{\Sigma_{\delta'}}+\frac{1}{2}\Sc(\Sigma_{\delta'},g)\geq0$ is equal to zero and the geodesic flow can be extended through $\delta'$.

By the above the geodesic flow starting from $\Sigma_0$ exists until it reaches $\p_+X$. Let $\{\Sigma_s\}_{0\leq s\leq d}$ be a maximal solution. As $\Sigma_d$ touches $\p_+X$ and $\Omega_d$ is a minimizer, the strong maximum principle Proposition \ref{prop:strongmax} implies that $\Sigma_d$ coincides with a connected component of $\p_+X$.

If we started out with $\Sigma_0\subset\p_-X$, then $X=\Sigma_0\times[0,d]$ since $X$ is connected. Furthermore \eqref{eq:metricproduct} holds true and we are done.
If $\Sigma_0\subset\mathring{X}$ or $\Sigma_0\subset\p_+X$ we consider the foliation $\Sigma_s$ for $s\in(-\delta,0]$.
We repeat the argument above (some signs change) to see that the geodesic flow in the direction of $-\nu$ exists until it meets $\p_-X$.
By the strong maximum principle and since $X$ is connected we conclude $X=\Sigma_0\times[c,d]$.
As \eqref{eq:metricproduct} holds this concludes the proof.
\end{proof}

Theroem \ref{thm:comparisongeneralB} implies part \eqref{itm:B} of Theorem \ref{MainTheorem}.
The last ingredient we need is the following observation appears in \cite[Theorem 2.3]{GroLa80} and is attributed to J. P. Bourguignon:

\begin{Proposition}
Let  $\Sigma$ be closed connected Riemannian manifold. If $\Sigma$ does not admit a metric with positive scalar curvature and $g$ is a Riemannian metric on $\Sigma$ with $\Sc(\Sigma,g)\geq0$, then $(\Sigma,g)$ is Ricci flat.
\end{Proposition}

\begin{Bemerkung}\label{rem:rigidity2}
The rigidity analysis in the Proof of Theorem \ref{thm:comparisongeneralB} could be adapted to prove rigidity of Theorem \ref{thm:comparisongeneralA} in case $\wid(X,g)=\wid(M,g_\varphi)$ if there was a way to guarantee the existence of a $\mu$-bubble in this situation. This is connected to Remark \ref{rem:rigidity} and Remark \ref{rem:weakness}.
\end{Bemerkung}

\begin{Bemerkung}\label{rem:2d2}
As we already alluded to in Remark \ref{rem:2d}, Theorem \ref{thm:comparisongeneralA} and Theorem \ref{thm:comparisongeneralB} apply to any oriented band $X$ in dimension $n=2$, as for any closed hypersurface $\Sigma$ (a collection of circles), which separates $\p_\pm X$, the operator $-\Delta_\Sigma+\frac{1}{2}\Sc(\Sigma,g)=-\Delta_\Sigma$ has first eigenvalue equal to zero.
\end{Bemerkung}

\section{Concerning separating hypersurfaces}\label{section:topology}

\begin{Lemma}\label{lem:sep}
Let $Y^{n-1}$ be a closed connected oriented manifold and $X=Y\times[-1,1]$. If $\Sigma$ is a closed embedded hypersurface in $X$, which separates $\p_-X$ and $\p_+X$, there is one connected component $\Sigma_0$ of $\Sigma$ that separates $\p_-X$ and $\p_+X$. 
\end{Lemma}

\begin{proof}
Without loss of generality $\Sigma$ can be assumed to be oriented since nonorientable components are non-separating.
Furthermore we can assume that $\Sigma\subset\mathring{X}$ (otherwise we isotope $\Sigma$ by flowing along the interior unit normal vector field to $\p X$ for a short time).
The relative homology group $H_1(X,\p X)$ is generated by paths $\gamma:[0,1]\rightarrow X$ with $\gamma(0)\in\p_- X$ and $\gamma(1)\in\p_+ X$.
Since the hypersurface $\Sigma$ separates $\p_-X$ and $\p_+X$ it has nonzero algebraic intersection with every such path $\gamma$.
It follows by Lefschetz duality that $[\Sigma]\neq0\in H_{n-1}(X)\cong H_{n-1}(Y)=\Z$.
Of course $[\Sigma]$ is nothing but $[\Sigma_0]+\ldots+[\Sigma_m]$, where $\Sigma_i$ are the connected components of $\Sigma$.
Since $[\Sigma]\neq0$ it  follows that $[\Sigma_i]\neq0$ for some $i\in\{0,\ldots,m\}$ (w.l.o.g. we can assume $[\Sigma_0]\neq0$).
Going back, by Lefschetz duality, $\Sigma_0$ has nonzero algebraic intersection with any path $\gamma$ which connects $\p_-X$ and $\p_+X$ and therefore it separates $\p_-X$ and $\p_+X$.
\end{proof}

\begin{Lemma}\label{lem:proper_separating}
Let \(\Sigma \subset X\) be a separating hypersurface in a band \(X\). 
Then there exists a union of components of \(\Sigma\) which is a properly separating hypersurface in \(X\).
\end{Lemma}
\begin{proof}
  Suppose that \(\Sigma\) is a separating hypersurface that contains a component not connected to both \(\p_- X\) and \(\p_+ X\) inside \(X \setminus \Sigma\). 
  Then the hypersurface \(\Sigma'\) obtained from \(\Sigma\) by deleting this component is still a separating hypersurface.
  This shows that there is a minimal collection of components of \(\Sigma\) such that its union is still separating yields the desired properly separating hypersurface.
\end{proof}

\begin{Lemma}\label{lem:degsep}
Let $X^n$ be a connected oriented band and $X'=Y^{n-1}\times[-1,1]$, where $Y$ is a closed connected oriented manifold.
Let $f:X\to X'$ be a band map with $\deg(f)=d\neq0$ and $\Sigma^{n-1}$ be a closed embedded hypersurface in $X$, which separates $\p_-X$ and $\p_+X$.
Then there is one connected component $\Sigma_0$ of  $\Sigma$ such that the map $(\pr_Y\circ f):\Sigma_0\to Y$ has nonzero degree.
\end{Lemma}

\begin{proof}
Without loss of generality $\Sigma$ can be assumed to be oriented, since nonorientable components are non-separating.
Furthermore we can assume that $\Sigma\subset\mathring{X}$ (otherwise we isotope $\Sigma$ by flowing along the interior unit normal vector field to $\p X$ for a short time).

By Lemma \ref{lem:proper_separating} there is a union of components of $\Sigma$ which is a properly separating hypersurface.
We denote this union of components by $\Sigma'$.
By construction every path $\gamma:[0,1]\to X$ with $\gamma(0)\in \p_-X$ and $\gamma(1)\in\p_+X$ has algebraic intersection number equal to one with $\Sigma'$.

Since $f$ is a band map $f\circ\gamma$ connects $\p_-X'$ and $\p_+X'$. It follows that $[\Sigma']$ is Lefschetz dual to $f^*\alpha$, where $\alpha$ is the generator of $H^1(X',\p X')\cong \Z$.

Consider the diagram:
\begin{displaymath}\begin{tikzcd}[row sep=large, column sep=large]
H^1(X,\partial X;\Z)\arrow[r, "\cap{[X,\partial X]}","\cong"'] 
& H_{n-1}(X;\Z)\arrow[d,"f_*"]\\
H^1(X',\p X';\Z)\arrow[u, "f^*"]\arrow[r,"\cap d{[X',\p X']}"] & H_{n-1}(X';\Z)\arrow[r,"\pr_{Y*}"] & H_{n-1}(Y;\Z).
\end{tikzcd}
\end{displaymath}
We conclude that $(\pr\circ f)_*[\Sigma']=d[Y]$.
Hence there is one connected component $\Sigma_0$ of $\Sigma'$ with $(\pr\circ f)_*[\Sigma']\neq0$. By construction $\Sigma_0$ is also a component of $\Sigma$.\end{proof}

\begin{Proposition}\label{prop:cobordism}
Let $Y$ be a closed connected oriented manifold of dimension $n-1\geq5$ and $X=Y\times[-1,1]$. Let $\Sigma_0$ be a closed connected oriented hypersurface separating $\p_-X$ and $\p_+X$ in $X$. If $\Sigma_0$ admits a metric with positive scalar curvature, then so does $Y$.
\end{Proposition}

\begin{proof}
The proof uses standard results and ideas from high dimensional topology.
We can assume that $\Sigma_0\subset\mathring{X}$ (otherwise we isotope $\Sigma_0$ by flowing along the interior normal vector field to $\p X$ for a short time).

We want to see that $Y$ can be obtained from $\Sigma_0$ by a finite sequence of surgeries in codimension $\geq3$ and hence, by the well known argument of Gromov and Lawson \cite[Theorem A]{GroLa}, a positive scalar curvature metric on $\Sigma_0$ can be transported to $Y$.
See \cite{EF18} for full details of the proof of \cite[Theorem A]{GroLa}.

We denote by $W$ the connected component of $X\backslash\Sigma_0$ which contains $\p_-X$. Then $W$ is a cobordism $W:Y\rightsquigarrow \Sigma_0$. 
We restrict the projection $X\rightarrow Y$ to $W$ and obtain a retract map $r:W\rightarrow Y$. 

\begin{claim}
The cobordism $W:Y\rightsquigarrow \Sigma_0$ and the retract map $r:W\rightarrow Y$ can be improved via surgery in the interior of $W$ to a cobordism $W_2:Y\rightsquigarrow \Sigma_0$ with a retract map $r_2:W_2\rightarrow Y$, which is 3-connected.
The inclusion $\iota:Y\hookrightarrow W_2$ will be 2-connected since $\iota\circ r_2=id_M$.
\end{claim}

\begin{proof}[Proof of Claim]
If $\nu(Y)$ denotes the stable normal bundle of $Y$, there is a stable trivialization of $r^*\nu(Y)\oplus TW$.
Since $r$ is a retract map the induced map $\pi_1(r):\pi_1(W)\rightarrow\pi_1(Y)$ is already surjective and its kernel is finitely generated as a normal subgroup of $\pi_1(W)$, since $\pi_1(W)$ is finitely generated and $\pi_1(Y)$ is finitely presented (see \cite[Lemma 3.2]{SZ20}). 
Let $\alpha$ be a generator of $\ker(\pi_1(r))$. Since $1<n/2$ we can represent $\alpha$ by an embedding $S^1\hookrightarrow \mathring{W}$. 
Since $W$ is oriented the normal bundle of this embedding is trivial and hence we can kill $\alpha$ by surgery in the interior of $W$. 
We obtain a cobordism $W_\alpha:Y\rightsquigarrow\Sigma_0$ and a retract map $r_\alpha:W_\alpha\rightarrow Y$. 
After repeating this step finitely many times we end up with $W_1:Y\rightsquigarrow\Sigma_0$ and a retract map $r_1:W_1\rightarrow Y$, which is 2-connected.

Next we need to kill the kernel of $\pi_2(r_1):\pi_2(W_1)\rightarrow\pi_2(Y)$. 
In order to do so one has to argue that this is possible with finitely many surgeries along elements of $\ker(\pi_2(r_1))$.
We proceed similarily as in the proof of \cite[Proposition 3.1]{SZ20}, which in turn is based on \cite[Lemma 5.6]{St98} and \cite[Lemma 1.1]{Wa65}. 
Since $Y$ and $W_1$ are compact manifolds, if one starts with a handle decomposition of $W_1$ relative to $Y$ one can use handle cancellation \cite{Wa71} and the fact that $Y\hookrightarrow W_1$ induces an isomorphism on $\pi_0$ and $\pi_1$ to get rid of 0-handles and 1-handles. 
All the (finitely many) 2-handles in this new handlebody are attached to $Y$ via contractible maps (otherwise they would kill elements in $\pi_1$). 
Hence the 2-skeleton $(W_1,Y)^{(2)}$ arising from this new handlebody is homotopy equivalent to $Y\vee (\bigvee_{j\in J}S^2)$. 
The 2-spheres in this wedge product finitely generate $\ker(\pi_2(r_1))$ as a $\Z[\pi_1(Y)]$-module over the common fundamental group $\pi_1(Y)=\pi_1(W_1)$.

Since $2< n/2$ we can represent each of those generators by an embedding $f:S^2\hookrightarrow \mathring{W_1}$. Since $r\circ f (S^2)$ is contractible, there is a map $g:D^3\rightarrow Y$ such that the following diagram commutes:
$$\begin{tikzcd}
S^2\arrow[d,"f"]\arrow[r,"i"]
&D^3 \arrow[d,"g"]\\
W_1\arrow[r,"r"]&Y.
\end{tikzcd}$$
The stable trivialization of $r^*\nu(Y)\oplus TW_1$ induces a stable trivialization of $f^*r^*\nu(Y)\oplus f^*TW_1=i^*g^*\nu(Y)\oplus f^*TW_1$. 
But $i^*g^*\nu(Y)$ is trivial since $D^3$ is contractible and hence $f^*TW_1\cong\nu(S^2,W_1)\oplus TS^2$ is stably trivial. 
Since $TS^2$ is stably trivial it follows that $\nu(S^2,W_1)$ is stably trivial and since $2<(n-1)/2$ we conclude that $\nu(S^2,W_1)$ is trivial.

Hence we can kill $\ker(\pi_2(r_1))$ in finitely many surgery steps. 
We end up with a cobordism $W_2:Y\rightsquigarrow \Sigma$ and a retract map $r_2:W_2\rightarrow Y$ which is 3-connected. 
Consequently the inclusion $\iota:Y\hookrightarrow W_2$ is 2-connected.
\end{proof}
If we start with a handle decomposition of $W_2$ with respect to $Y$ we can use handle cancellation  \cite{Wa71} to get rid of all the $0$-, $1$- or $2$-handles since the inclusion $\iota:Y\hookrightarrow W_2$ is 2-connected.
Turning this upside down this handle decomposition can be interpreted as a handle decomposition of $W_2$ with respect to $\Sigma_0$. 
In this interpretation the dimension of each handle becomes its codimension.

Consequently $W_2$ can be obtained from $\Sigma_0\times[-1,1]$ by attaching handles of codimension $\geq 3$ and $Y$ can be obtained from $\Sigma_0$ by a finite sequence of surgeries in codimension $\geq3$.
Thus, by \cite[Theorem A]{GroLa}, $Y$ admits a metric of positive scalar curvature if $\Sigma_0$ does.
\end{proof}

We have all the ingredients to prove the main results of Section \ref{section:examples}. 
Proposition \ref{prop:cobordismalt} follows directly from Lemma \ref{lem:sep} and Proposition \ref{prop:cobordism}. Further Proposition \ref{prop:NPSC} follows directly from Lemma \ref{lem:sep} and Definition \ref{def:NPSC}.
For the convenience of the reader we also include a proof of Proposition \ref{prop:rosenberg}, which heavily draws on the work of Zeidler \cites{RZ19,RZ20}.

\begin{proof}[Proof of Proposition \ref{prop:rosenberg}]
By Lemma \ref{lem:sep} there is one connected component $\Sigma_0$ of $\Sigma$, which separates $\p_-X$ and $\p_+X$. 
We can assume that $\Sigma_0\subset\mathring{X}$ (otherwise we isotope $\Sigma_0$ by flowing along the interior normal vector field to $\p X$ for a short time).
We consider the real Mi\v{s}\v{c}enko bundle $\mathcal{L}_Y\rightarrow Y$, which is the flat bundle of finitely generated projective Hilbert-$C^*\pi_1Y$-modules associated to the representation of $\pi_1Y$ on $C^*\pi_1Y$ by left multiplication. 
Recall (see for example \cite{RZ20}*{Section 2}) that the Rosenberg index $\alpha(Y)\in KO_{n-1}(C^*\pi_1Y)$ is then the (K-theoretic) index of the Dirac operator on the spinor bundle of $Y$ twisted with $\mathcal{L}_Y$. 
We pull back $\mathcal{L}_Y$ to $X$ via the projection $X\rightarrow Y$ and restrict this pullback bundle to the connected component $W$ of $X$ which is bounded by $\p_-X$ and $\Sigma_0$. 
We denote the resulting bundle by $\mathcal{E}\rightarrow W$.

Since $Y$ is spin, so are $W$ and $\Sigma_0$. If we restrict $\mathcal{E}$ to $\Sigma_0$, the index of the Dirac operator on the spinor bundle of $\Sigma_0$ twisted with the restriction of $\mathcal{E}$ is an element in $KO_{n-1}(C^*\pi_1Y)$ which we denote by $\alpha_\mathcal{E}(\Sigma_0)$. 
By bordism invariance of the index $\alpha_\mathcal{E}(\Sigma_0)=\alpha(Y)\neq 0$ and hence $\Sigma_0$ does not admit a metric with positive scalar curvature as $\mathcal{E}$ is a flat bundle and by the usual argument involving the Lichnerowicz-Weitzenb\"ock formula.
\end{proof}

\section{A codimension two obstruction}\label{section:last}

In this section we present a detailed proof of Theorem \ref{prop:deg}.
To do so we implement the ideas of \cite{Gro20}*{Section 7, Main Theorem} using the techniques developed in \cite{ChLi20}.

\begin{Bemerkung}
To unburden the notation in this section we will denote the scalar curvature of a Riemannian manifold $(M,g)$ by $R_M$. If $B\subset M$ is an embedded submanifold we will denote the scalar curvature of the induced metric $g\bigr|_B$ by $R_B$.
\end{Bemerkung}

\begin{Definition}\label{def:triple}
Let $M^n$ be a band. Let $\alpha\neq0\in H_{n-2}(M;\Z)$ be a non torsion homology class. We say that $\alpha$ is a \emph{band class} if there are $\alpha^+\in H_{n-2}(\p_+M;\Z)$ and $\alpha^-\in H_{n-2}(\p_-M;\Z)$ with $\alpha=\iota_*(\alpha^\pm)$, where $\iota:\p M\rightarrow M$ denotes the inclusion of the boundary.
\end{Definition}

The main analytical tool we need to develop is the following proposition, which is reminiscent of what Gromov, in \cite{Gro20}*{Section 3}, calls \emph{Richard's Lemma} in reference to \cite{Ri20}. The proof, however, follows in the line of  \cite{ChLi20}*{Sections 6.1 \& 6.2}.

\begin{Proposition}\label{prop:ana}
Let $(M^4,g)$ be a Riemannian band and $\alpha\in H_{2}(M;\Z)$ be a band class. If $R_M>\sigma>0$ and $\wid(M,g)>\frac{2\pi}{\sqrt{\sigma}}$, there is a smooth oriented submanifold $\Sigma$ which represents $\alpha$ and each connected component $\Sigma_0$ of $\Sigma$ is homeomorphic to a 2-sphere with $$\diam(\Sigma_0,g|_{\Sigma_0})\leq\pi\sqrt{\frac{2}{\inf R_M-\sigma}}.$$
\end{Proposition}

\begin{proof}
Denote $\ell=\frac{\pi}{\sqrt{\sigma}}$. Let $\beta\in H_{3}(M,\p M;\Z)$ be the relative class with $\p\beta=\alpha^+-\alpha^-$. Let $B^{3}$ be a free boundary minimal hypersurface in the class $\beta$.

By stability of $B$ and the classical second variation formula for the area functional we see that
\begin{displaymath}
\int_B|\nabla_B\psi|^2-\frac{1}{2}(R_M-R_B+|A|^2)\psi^2d\Ha^{3}\geq0,
\end{displaymath}
for all $\psi\in C^1_0(B)$. As such, there is a function $u\in C^\infty_0(B$) with $u>0$ on $\mathring{B}$ and
\begin{equation}\label{eq:1ev}
\Delta_Bu\leq-\frac{1}{2}(R_M-R_B+|A|^2)u.
\end{equation}
Furthermore $(B,g|_B)$ is a Riemannian band with $\wid(B,g|_B)>2\ell$. We can shrink $B$ a little bit from both sides such that $\wid>2\ell$ remains true but $\p B\subset \mathring{M}$ (this guarantees $u>0$ on $B$). Let $\phi:B\rightarrow[-\ell,\ell]$ be the map produced by Lemma \ref{lem:map} and set $h(x)=-\frac{\pi}{\ell}\tan(\frac{\pi}{2\ell}\phi(x))$.

We define $\Omega_0=\phi^{-1}[-\ell,0]$ and consider the functional 
\begin{displaymath}
\mathcal{A}(\hat{\Omega})=\int_{\p^*\hat{\Omega}}ud\Ha^{2}-\int_B(\chi_{\hat{\Omega}}-\chi_{\Omega_0})hud\Ha^3
\end{displaymath}
for all Caccioppoli sets $\hat{\Omega}$ with $\hat{\Omega}\Delta\Omega_0$ contained in the interior of $B$.
By Lemma \ref{lem:ex} we find a $\mu$-bubble $\Omega\subset B$ with smooth boundary $\Sigma=(\p\Omega\backslash\p_-B)$, which represents the class $\alpha\in H_{2}(M;\Z)$.
Stability of $\Omega$ and Lemma \ref{lem:var2} imply that for each connected component $\Sigma_0$ of $\Sigma$ we have:
\begin{displaymath}
\int_{\Sigma_0}|\nabla_{\Sigma_0}\psi|^2u-\frac{1}{2}\left(R_B-R_{\Sigma_0}-H^2+|A|^2\right)\psi^2u+\left(2H\langle\nabla_Bu,\nu\rangle+\frac{d^2u}{d\nu^2}-Hhu-\langle\nabla_B(hu),\nu\rangle\right)\psi^2\\
\geq0
\end{displaymath}
We use:
\begin{displaymath}
\frac{d^2u}{d\nu^2}+H\langle\nabla_Bu,\nu\rangle=\Delta_Bu-\Delta_{\Sigma_0}u\leq-\frac{1}{2}(R_M-R_B+|A|^2)u-\Delta_{\Sigma_0}u.
\end{displaymath}
and as a consequence of Lemma \ref{lem:var1}
\begin{displaymath}
Hhu=H^2u+H\langle\nabla_Bu,\nu\rangle,
\end{displaymath}
and 
\begin{displaymath}
\frac{1}{2}H^2\psi^2u=\frac{1}{2}u^{-1}\langle\nabla_Bu,\nu\rangle^2\psi^2-h\langle\nabla_Bu,\nu\rangle+\frac{1}{2}h^2\psi^2u.
\end{displaymath}
Hence
\begin{displaymath}
\begin{split}
0&\leq \int_{\Sigma_0}|\nabla_{\Sigma_0}\psi|^2u-\frac{1}{2}(R_M-R_{\Sigma_0}+H^2+2|A|^2)\psi^2u-(\Delta_{\Sigma_0}u+h\langle\nabla_B(u),\nu\rangle+u\langle\nabla_Bh,\nu\rangle)\psi^2\\
&=\int_{\Sigma_0}|\nabla_{\Sigma_0}\psi|^2u-\frac{1}{2}(R_M-R_{\Sigma_0}+2|A|^2)\psi^2u-(\Delta_{\Sigma_0}u)\psi^2-\frac{1}{2}(h^2+2\langle\nabla_Bh,\nu\rangle)\psi^2u\\
&\leq\int_{\Sigma_0}|\nabla_{\Sigma_0}\psi|^2u-\frac{1}{2}(R_M-\sigma-R_{\Sigma_0})\psi^2u-(\Delta_{\Sigma_0}u)\psi^2-\frac{1}{2}(\sigma+h^2+2\langle\nabla_Bh,\nu\rangle)\psi^2u
\end{split}
\end{displaymath}
and since, using $\ell=\frac{\pi}{\sqrt{\sigma}}$, we can estimate
\begin{displaymath}
\sigma+h^2+2\langle\nabla_Bh,\nu\rangle\geq\sigma+h^2-2|\nabla_Bh|=0.
\end{displaymath}
We conclude
\begin{equation}\label{eq:star}
0<\int_{\Sigma_0}|\nabla_{\Sigma_0}\psi|^2u-\frac{1}{2}(R_M-\sigma-R_{\Sigma_0})\psi^2u-(\Delta_{\Sigma_0}u)\psi^2.
\end{equation}
If we choose $\psi=u^{-\frac{1}{2}}$, this yields
\begin{displaymath}
\begin{split}
0&<\int_{\Sigma_0}|\nabla_{\Sigma_0}u^{-\frac{1}{2}}|u-\frac{1}{2}(R_M-\sigma-R_{\Sigma_0})-(\Delta_{\Sigma_0}u)u^{-1}\\
&=\int_{\Sigma_0}-\frac{3}{4}u^{-2}|\nabla_{\Sigma_0}u|^2-\frac{1}{2}(R_M-\sigma-R_{\Sigma_0})\\
&\leq\int_{\Sigma_0}-\frac{1}{2}(R_M-\sigma-R_{\Sigma_0}),
\end{split}
\end{displaymath}
or equivalently 
\begin{displaymath}
\frac{1}{2}\int_{\Sigma_0}R_M-\sigma\leq\frac{1}{2}\int_{\Sigma_0}R_{\Sigma_0}=2\pi\chi(\Sigma_0),
\end{displaymath}
which implies that $\Sigma_0$ is a sphere with $\area(\Sigma_0) \leq\frac{2\pi}{\inf R_M-\sigma}$.

To finish the proof we return to \eqref{eq:star} and see that it implies the existence of a positive function $w\in C^\infty(\Sigma_0)$, with
\begin{displaymath}
\dive_{\Sigma_0}(u\nabla_{\Sigma_0}w)\leq-\frac{1}{2}(R_M-\sigma-R_{\Sigma_0})wu-(\Delta_{\Sigma_0}u)w.
\end{displaymath}
If we set $\lambda=uw$, then 
\begin{displaymath}
\begin{split}
\Delta_{\Sigma_0}\lambda&=\dive_{\Sigma_0}(u\nabla_{\Sigma_0}w)+\dive_{\Sigma_0}(w\nabla_{\Sigma_0}u)\\
&\leq-\frac{1}{2}(R_M-\sigma-R_{\Sigma_0})\lambda-(\Delta_{\Sigma_0}u)w+\dive_{\Sigma_0}(w\nabla_{\Sigma_0}u)\\
&\leq-\frac{1}{2}(R_M-\sigma-R_{\Sigma_0})\lambda+\langle\nabla_{\Sigma_0}u,\nabla_{\Sigma_0}w\rangle\\
&\leq-\frac{1}{2}(R_M-\sigma-R_{\Sigma_0})\lambda+\frac{1}{2}\lambda^{-1}|\nabla_{\Sigma_0}\lambda|^2.
\end{split}
\end{displaymath}
Now $\diam(\Sigma_0,g|_{\Sigma_0})\leq\pi\sqrt{\frac{2}{\inf R_M-\sigma}}$ follows directly from the next lemma.
\end{proof}

\begin{Lemma}[\cite{ChLi20}*{Lemma 16}]\label{lem:surface}
Let $(N^2,g)$ be a closed 2-dimensional Riemannian manifold. If there is a smooth function $\lambda>0$ on $\Sigma_0$ with 
$$
\Delta_{N}\lambda\leq-\frac{1}{2}(C-R_{N})\lambda+\frac{1}{2}\lambda^{-1}|\nabla_{N}\lambda|^2
$$
for some $C>0$, then $\diam(N,g)\leq\sqrt{\frac{2}{C}}\pi$.
\end{Lemma}

\begin{proof}
Let $\ell=\sqrt{\frac{2}{C}}\pi$. Assume for a contradiction, that there are $p,q\in N$ with $\dist_g(p,q)>\ell$. For $\e>0$ small enough $M=N\backslash (B_\e(p)\cup B_\e(q))$ is a band with $\wid(M,g)>\ell$. Let $\phi:M\rightarrow[-\frac{\ell}{2},\frac{\ell}{2}]$ be the map we get from Lemma \ref{lem:map} and define $h(x)=-\frac{2\pi}{\ell}\tan(\frac{\pi}{\ell}\phi(x))$.\\
Let $\Omega_0=\phi^{-1}[-\frac{\ell}{2},0]$ (w.l.o.g 0 is a regular value of $\phi$) and consider the functional 
\begin{displaymath}
\mathcal{A}(\hat{\Omega})=\int_{\p^*\hat{\Omega}}\lambda-\int_M(\chi_{\hat{\Omega}}-\chi_{\Omega_0})h\lambda d\Ha^2
\end{displaymath}
for all Caccioppoli sets $\hat{\Omega}$ with $\hat{\Omega}\Delta\Omega_0$ contained in the interior of $M$.

By Lemma \ref{lem:ex} there is a minimizer $\Omega$ with smooth boundary $(\p\Omega\backslash\p_-M)$ and by Lemma \ref{lem:var1} every connected component $\Sigma$ satisfies
\begin{displaymath}
H=-\lambda^{-1}\langle\nabla_M\lambda,\nu\rangle+ h.
\end{displaymath}
Note that in this case $H$ is the geodesic curvature.\\
By stability and Lemma \ref{lem:var2} we see
\begin{displaymath}
0\leq\int_\Sigma|\nabla_\Sigma\psi|^2\lambda-\frac{1}{2}(R_M-H^2+|A|^2)\psi^2\lambda+(2H\langle\nabla_M\lambda,\nu\rangle+\frac{d^2\lambda}{d\nu^2}-Hh\lambda-\langle\nabla_M(h\lambda),\nu\rangle)\psi^2
\end{displaymath}
for all $\psi\in C^1(\Sigma)$. We use:
\begin{displaymath}
\frac{d^2\lambda}{d\nu^2}+H\langle\nabla_Mu,\nu\rangle=\Delta_M\lambda-\Delta_{\Sigma}\lambda\leq-\Delta_{\Sigma}\lambda-\frac{1}{2}(C-R_M)\lambda+\frac{1}{2}\lambda^{-1}|\nabla_M\lambda|^2,
\end{displaymath}
as a consequence of Lemma \ref{lem:var1}
\begin{displaymath}
Hh\lambda=H^2\lambda+H\langle\nabla_M\lambda,\nu\rangle,
\end{displaymath}
and 
\begin{displaymath}
\frac{1}{2}H^2\psi^2\lambda=\frac{1}{2}\lambda^{-1}\langle\nabla_M\lambda,\nu\rangle^2\psi^2-h\langle\nabla_M\lambda,\nu\rangle+\frac{1}{2}h^2\psi^2\lambda.
\end{displaymath}
It follows that
\begin{displaymath}
0\leq\int_\Sigma|\nabla_\Sigma\psi|^2\lambda+\frac{1}{2}\lambda^{-1}|\nabla_M\lambda|^2\psi^2-\frac{1}{2}\lambda^{-1}\langle\nabla_M\lambda,\nu\rangle^2\psi^2-(\Delta_{\Sigma}\lambda)\psi^2-\frac{1}{2}(C+h^2+2\langle\nabla_Mh,\nu\rangle)\psi^2\lambda
\end{displaymath}
and since $C+h^2+2\langle\nabla_Mh,\nu\rangle>0$ (remember that $\Lip(\phi)<1$) by construction of $h$, we see
\begin{displaymath}
0<\int_\Sigma|\nabla_\Sigma\psi|^2\lambda+\frac{1}{2}\lambda^{-1}|\nabla_M\lambda|^2\psi^2-\frac{1}{2}\lambda^{-1}\langle\nabla_M\lambda,\nu\rangle^2\psi^2-(\Delta_{\Sigma}\lambda)\psi^2\\
\end{displaymath}
If we choose $\psi=\lambda^{-\frac{1}{2}}$ this yields 
\begin{displaymath}
\begin{split}
0&<\int_\Sigma -\frac{3}{4}\lambda^{-2}|\nabla_\Sigma\lambda|^2+\frac{1}{2}\lambda^{-2}|\nabla_M\lambda|^2-\frac{1}{2}\lambda^{-2}\langle\nabla_M\lambda,\nu\rangle^2\\
&=\int_\Sigma -\frac{1}{4}\lambda^{-2}|\nabla_\Sigma\lambda|^2,
\end{split}
\end{displaymath}
which is a contradiction.
\end{proof}

\begin{Definition}\label{def:fill}
Let $(X^n,g)$ be a complete oriented Riemannian manifold and $c$ a locally finite singular $k$-cycle. Then the the \emph{filling radius} of $c$ in $(X,g)$ is defined to be
\begin{displaymath}
\fil_\Z(c,X)=\inf\{r>0 | [c]=0\in H_k^{lf}(U_r(c);\Z)\},
\end{displaymath}
where $U_r(c)$ denotes the open $r$-neighborhood of $c$ in $(X,g)$. If one replaces $\Z$ by $\Q$ the same definition yields the rational filling radius.
\end{Definition}

\begin{Bemerkung}\label{rem:fill}
In \cite{Gro83}*{Section 1} Gromov generalizes the above and defines the filling radius $\fil_\Z(X,g)$ of a complete oriented Riemannian manifold $(X,g)$, with respect to the Kuratowski embedding of $(X,d_g)$. Furthermore he proves two results we will need in the following:
\begin{enumerate}
\item $\fil_\Z(\R,g_{std})=\infty$ (see \cite{Gro83}*{Section 4.4.C})
\item if $X$ is isometrically embedded in a Metric space $(S,d)$, then $\fil_\Z(X,S)\geq \fil_\Z(X,g)$ (see \cite{Gro83}*{Section 1}).
\end{enumerate}
Both points remain true with rational coefficients.
\end{Bemerkung}

\begin{Lemma}[\cite{Iv86}*{Theorem IX.4.7}]\label{lem:alex}
Let $X^n$ be an oriented manifold. Let $C\subset X$ be a closed subset such that $\p C=C\backslash\mathring{C}$ is a smooth submanifold. Then $H_1^{lf}(C;\Z)\cong H^{n-1}(X,X\backslash C;\Z)$.
\end{Lemma}

\begin{Lemma}[Codim 2 Linking Lemma \cite{Gro19}*{Lemma 4.G}]\label{lem:link}
Let $Y^n$ be a closed aspherical manifold and $g$ a Riemannian metric on $Y$. For every $\sigma>0$ there is a compact band $M$ in the universal cover $(\ti{Y},\ti{g})$ such that: $\wid(M,\ti{g})>\frac{2\pi}{\sqrt{\sigma}}$, there is a band class $\alpha\in H_{n-2}(M;\Z)$ and for every cycle $c\subset M$ representing a nonzero multiple of $\alpha$ we have $\fil_\Z(c,\ti{Y})>\frac{2\pi}{\sqrt{\sigma}}$.
\end{Lemma}

\begin{proof}
Let $\sigma>0$ be arbitrary. By \cite{ChLi20}*{Lemma 6} there is a geodesic line $\gamma:\R\rightarrow(\ti{Y},\ti{g})$.
Since $\gamma$ is an isometric embedding of the real line, $\fil_\Z(\gamma,\ti{Y})\geq\fil_\Z(\R,g_{Std})=\infty$ (see Remark \ref{rem:fill}), hence  for all $r>0$ the line $\gamma$ represents a non-zero class in $H_1^{lf}(U_r(\gamma),\Z)$, where $U_r(\gamma)$ denotes the open $r$-neighborhood of $\gamma$ in $(\ti{Y},\ti{g})$. Since $\fil_\Q(\gamma,\ti{Y})\geq\fil_\Q(\R,g_{Std})=\infty$ we see by the  same argument, that $[\gamma]$ is non torsion in $H_1^{lf}(U_r(\gamma),\Z)$.

For some $\e>0$ let $\rho$ be a smooth approximation of $\dist(\gamma,\cdot)$ which is $\e$-close and such that $\Lip(\rho)<1+\e$. There is a sequence $(r_k)_{k\in\mathbb{N}}$ of regular values of $\rho$ with $r_k\rightarrow\infty$ and the property that for all $k\in \mathbb{N}$ we have $r_k>2\e$ and $r_{k+1}-r_k>2\e$. Denote $\ov{U}_k=\rho^{-1}[0,r_k]$. By construction $[\gamma]\neq0\in H_1^{lf}(\ov{U}_k,\Z)$ for all $k$ and the class is non torsion. Thus by Lemma \ref{lem:alex}, the fact that $\ti{Y}$ is contractible we see:
\begin{displaymath}
0\neq[\gamma]\in H_1^{lf}(\U_k,\Z)\cong H^{n-1}(\ti{Y},\ti{Y}\backslash \U_k;\Z)\cong H^{n-2}(\ti{Y}\backslash \U_k;\Z).
\end{displaymath}
Furthermore, since $[\gamma]$ is non torsion its image in $H^{n-2}(\ti{Y}\backslash \U_k;\Z)$ corresponds by the UCT to an element $\alpha_k\neq0\in H_{n-2}(\ti{Y}\backslash \U_k;\Z)$. If we represent $\alpha_k$ by a closed smooth submanifold $N_k\subset\ti{Y}\backslash \U_k$, then $N$ is linked with $\gamma$ and $\dist(\gamma,N_k)>r_k-\e>r_{k-1}+\e$.

Let $V_k$ be a smoothed version (as before) of the closed $(r_{k-1}/2)$-neighborhood of $N_k$. Then $0\neq\alpha_k\in H_{n-1}(V_k;\Z)$, since $N_k$ is linked with $\gamma$ and $V_{k}\subset(\ti{Y}\backslash\gamma)$. Furthermore any cycle $c\subset V_k$, which represents a nonzero multiple of $\alpha_k$ is linked with $\gamma$ (since $\alpha_k$ is non torsion) and hence $\fil_\Z(c,\ti{Y})\geq\dist(\gamma,c)\geq (r_{k-1}/2)$. We want to see that there is a class $\alpha_k^+\in H_{n-1}(\p V_{k};\Z)$ with $\iota(\alpha_k^+)=\alpha_k$, where $\iota:\p V_{k}\rightarrow V_{k}$ denotes the inclusion of the boundary.

Consider the $\e$-neighborhood $U_\e(\gamma)$ of $\gamma$. There is a closed smooth submanifold $N_0\subset U_\e(\gamma)$ with $[N_0]=[N_k]\in H_{n-2}(\ti{Y}\backslash\gamma;\Z)$ (this $N_0$ will be the image under the exponential map of small sphere around the origin in the fiber over $\gamma(0)$ in the normal bundle of $\gamma$). Hence we can find a smooth oriented submanifold $B^{n-1}$ with $\p B = N_0\cup N_k$ and by possibly deforming $B$ a little bit we can ensure that $B$ intersects $\p V_{k}$ transversely in a closed $(n-2)$-dimensional submanifold $N^+_k$. Then $[N_k^+]$ is homologous to $[N_k]$ in $H_{n-2}(V_{k};\Z)$ and we can set $\alpha_k^+=[N_k^+]\in H_{n-2}(\p V_{k};\Z)$.

Finally, for $\delta>0$ small enough, $U_\delta(N_k)$ is isometric via the exponential map to the normal $\delta$-disc bundle $D_\delta(N_k)$. It follows that $\p U_\delta(N_k)=N_k\times S^1$ and we set $\alpha_k^-=[N_k\times\{0\}]\in H_{n-2}(N_k\times S^1;\Z)$. We conclude that $M_k=V_{k}\backslash U_\delta(N_k)$ is a compact band with boundary $\p_+M_k=\p V_k$ and $\p_-M_k=\p U_\delta(N_k)=N_k\times S^1$, which has $\wid(M_k,\ti{g})\geq r_{k-1}-\e$. Furthermore the triple $(\alpha_k, \alpha_k^+, \alpha_k^-)$ represents a band class in $M_k$. For $k$ big enough $\wid(M_k,\ti{g})\geq \frac{r_{k-1}}{2}-\e>\frac{2\pi}{\sqrt{\sigma}}$.
\end{proof}

\begin{Lemma}\label{lem:hopf}
Let $M'$ and $M$ be two smooth compact connected oriented $n$-dimensional bands and $f:M'\rightarrow M$ be a map of degree $d\neq0$ with $f(\p_\pm M')=\p_\pm M$. If $\alpha\in H_{n-2}(M;\Z)$ is a band class, there is a band class $\alpha'\in H_{n-2}(M';\Z)$ with $f_*(\alpha')=d\alpha$.
\end{Lemma}

\begin{proof}
Consider the following diagram:
\begin{displaymath}\begin{CD}
H_{n-2}(M';\Z)@<{\cap[M',\p M']}<\cong<H^2(M',\p M';\Z)\\
@Vf_*VV				@AAf^*A\\
H_{n-2}(M;\Z)@<{\cap d[M,\p M]}<<H^2(M,\p M;\Z).
\end{CD}
\end{displaymath}
The diagram commutes since $f_*[M',\p M']=d[M,\p M]$. By Lefschetz duality there is a unique cohomology class $\eta\in H^2(M,\p M;\Z)$ with $\eta\cap d[M,\p M]=d\alpha\neq0$ since $\alpha$ is non torsion. Then $\alpha':=f^*\eta\cap[M',\p M']$ is such that $f_*\alpha'=d\alpha$.

The (co)homology of $\p M$ and $\p M'$ splits as the direct sum of the (co)homology of the components $\p_\pm M$ and $\p_\pm M'$ i.\ e. $H_*(\p M)=H_*(\p_-M)\oplus H_*(\p_+M)$  and  $H^*(\p M)=H^*(\p_-M)\oplus H^*(\p_+M)$. The induced maps $f_*$ and $f^*$ split into components as well. By comparing the components of
$$f_*([\p_+M'])-f_*([\p_-M'])=\p f_*[M',\p M']=\p d[M,\p M]=d[\p_+ M]-d[\p_-M],$$
we conclude that $f_*([\p_+M'])=d[\p_+ M]$ and $f_*([\p_-M'])=d[\p_- M]$, so the restricted maps between the boundary components have degree $d$ as well.

For both boundary components we separately write down a diagram as above
\begin{displaymath}\begin{CD}
H_{n-2}(\p_\pm M';\Z)@<{\cap[\p_\pm M']}<\cong<H^1(\p_\pm M';\Z)\\
@Vf_*VV				@AAf^*A\\
H_{n-2}(\p_\pm M;\Z)@<{\cap d[\p_\pm M]}<<H^1(\p_\pm M;\Z)
\end{CD}
\end{displaymath}
and we find unique cohomology classes $\eta^\pm\in H^1(\p_\pm M;\Z)$ with $\eta^\pm\cap d[\p_\pm M]=\pm d\alpha^\pm$.\\
We then consider 
\begin{displaymath}
\begin{CD}
H^1(\p_-M;\Z)\oplus H^1(\p_+M;\Z)@>{\cap[\p M]}>>H_{n-2}(\p_-M;\Z)\oplus H_{n-2}(\p_+M;\Z)\\
@VVV				@VV\iota_*V\\
H^2(M,\p M;\Z)@>{\cap[M]}>>H_{n-2}(M;\Z)
\end{CD}
\end{displaymath}
By comparing components we see that $\eta^-$ and $\eta^+$ map to $\eta$ under the connecting homomorphism $H^1(\p M;\Z)\to H^2(M,\p M;\Z)$. Thus we define classes $\alpha'^\pm:=f^*\eta^\pm\cap[\p_\pm M']$. Finally 
\begin{displaymath}
\begin{CD}
H^1(\p M;\Z)@>f^*>>H^1(\p M';\Z)@>{\cap[\p M']}>>H_{n-2}(\p M';\Z)\\
@VVV				@VVV @V\iota_*VV\\
H^2(M,\p M;\Z)@>f^*>>H^2(M',\p M';\Z)@>{\cap [M]}>>H_{n-2}(M';\Z)
\end{CD}
\end{displaymath}
implies that $\iota_*(\alpha'^\pm)=\alpha'$.
\end{proof}

The following Proposition is well known. A  detailed proof can be found in \cite{ChLiL21}*{Section 4}.

\begin{Proposition}\label{prop:fix}
Let $Y^n$ and $Z^n$be closed connected oriented manifolds and $f:Z\to Y$ a smooth map with $\deg(f)\neq0$.
The pullback $pr:\hat{Z}\rightarrow Z$ of the universal covering $\ti{Y}\to Y$ has the following properties:
\begin{itemize}
\item $\hat{Z}$ is non compact,
\item the map $f\circ pr:\hat{Z}\to Y$ can be lifted to a map $\hat{f}:\hat{Z}\rightarrow\ti{Y}$,
\item $\hat{f}$ is proper and $\deg(\hat{f})=\deg(f)$.
\end{itemize}
\end{Proposition}

\begin{proof}[Proof of Theorem \ref{prop:deg}]
Let $Y^4$ be a closed oriented aspherical manifold.
Let $Z^4$ be a closed oriented manifold and $f:Z\to Y$ a continuous map with $\deg(f)\neq0$. We remind the reader that for any metric on $Y$ the universal cover, equipped with the pullback metric, is uniformly contractible (see for example \cite{Gro83}*{Section 4.5.D}) i.\ e. for every radius $R>0$ there is a radius $C(R)>0$ such that for every point $y\in \ti{Y}$ the ball $B_R(y)$ is contractible within $B_{C(R)}(y)$.\\
Assume for a contradiction, that $Z$ admits a Riemannian metric $g_1$ with $Sc(Z,g)=R_Z>\sigma>0$. By possibly replacing it with a homotopic map, we can assume that $f:Z\rightarrow Y$ is smooth. Let $g_2$ be a Riemannian metric on $Y$. Then $f:(Z;g_1)\rightarrow (Y,g_2)$ is a Lipschitz map and by possibly rescaling $g_2$, we can assume that it is distance decreasing.

By Proposition \ref{prop:fix} there is a covering space $\hat{Z}$ of $Z$ and a lift $\hat{f}:\hat{Z}\rightarrow \ti{Y}$ such that $\hat{f}$ is proper and $\deg(\hat{f})=\deg(f)$. By Lemma \ref{lem:link} we can find a compact band $M$ in $(\ti{Y},\ti{g_2})$ with $\wid(M,\ti{g_2})>\frac{2\pi}{\sqrt{\sigma}}$ and a band class $\alpha\in H_{n-2}(M;\Z)$, such that for every cycle $c\subset M$ representing a nonzero multiple of $\alpha$ we have $\fil_\Z(c,\ti{Y})>\frac{2\pi}{\sqrt{\sigma}}$. By transversality we can deform the map $\hat{f}$ by an arbitrarily small amount to make it transverse to $\p M_k=\p_+M\cup\p_-M$, while remaining distance decreasing.\\
Then $M'= \hat{f}^{-1}(M)$ is a compact band in $\hat{Z}$ with smooth boundary $$\p M'=\hat{f}^{-1}(\p M)=\hat{f}^{-1}(\p_+ M)\cup\hat{f}^{-1}(\p_- M)=:\p_+M'\cup\p_-M'$$ and since $\hat{f}$ is distance decreasing $\wid(M',\hat{g_1})>\frac{2\pi}{\sqrt{\sigma}}$. Furthermore $\hat{f}$ restricts to a map $M'\rightarrow M$ of degree non-zero. By Lemma \ref{lem:hopf}, there is a band class $\alpha'\in H_{n-2}(M';\Z)$ with $\hat{f}_*\alpha'=\deg(\hat{f})\alpha\neq 0$.

By Proposition \ref{prop:ana} there is a smooth oriented submanifold $\Sigma$ which represents $\alpha'$ and  each connected component $\Sigma_0$ of $\Sigma$ is a 2-sphere with
$$\diam(\Sigma_0,\hat{g}|_{\Sigma_0})\leq\sqrt{\frac{2}{\inf R_Z-\sigma}}\pi.$$
Then $\hat{f}(\Sigma)$ is a cycle $c_0$ in $M$, which represents $\deg(\hat{f})\alpha$ and $\fil_\Z(c_0,\ti{Y})\leq C\left(\sqrt{\frac{2}{\inf R_Z-\sigma}}\pi\right)$, since $\hat{f}$ is distance decreasing and $(\ti{Y},\ti{g})$ is uniformly contractible. For $\sigma>0$ small enough this yields
\begin{equation}\label{eq:last}
\fil_\Z(c_0,\ti{Y})>\frac{2\pi}{\sqrt{\sigma}}>C\left(\sqrt{\frac{2}{\inf R_Z-\sigma}}\pi\right)\geq\fil_\Z(c_0,\ti{Y}),
\end{equation}
which is a contradiction.
\end{proof}

\end{document}